\documentclass[11pt,a4paper,leqno]{article}
\usepackage[utf8]{inputenc}
\usepackage[T1]{fontenc}
\usepackage[left=2.5 cm,right=2.5cm,top=2.5cm,bottom=2.5cm]{geometry}
\usepackage{amsthm,amssymb,amsmath,amsfonts,empheq}
\usepackage{mathtools}
\usepackage{color}
\usepackage{xcolor}
\usepackage{array}
\usepackage[shortlabels]{enumitem}
\usepackage{graphicx}
\usepackage{fancyhdr}
\usepackage[french,english]{babel}
\usepackage{hyperref}
\usepackage{eucal}
\usepackage{capt-of}
\usepackage{moreverb}

\newtheorem{theorem}{\textbf{Theorem}}[section]
\newtheorem{remark}[theorem]{\textbf{Remark}}

\newtheorem{lemma}[theorem]{\textbf{Lemma}}
\newtheorem{corollary}[theorem]{\textbf{Corollary}}
\newtheorem{proposition}[theorem]{\textbf{Proposition}}

\newtheorem{definition}[theorem]{\textbf{Definition}}

\numberwithin{equation}{section}

\usepackage{titlesec}
\titleformat\section{}{}{0pt}{\Large\scshape\filcenter\thesection{} - }

\def\aa{\alpha}
\def\bb{\beta}
\def\dd{\delta}

\def\g{\gamma}

\def\GG{\Gamma}
\def\ll{\lambda}
\def\LL{\Lambda}

\def\th{\theta}

\def\e{\varepsilon}
\def\ee{\varepsilon}
\def\vp{\varphi}

\newcommand{\aaa}{\boldsymbol{\alpha}}

\newcommand{\qo}{{\overline q}}

\newcommand{\aao}{{\overline \aa}}

\newcommand{\cM}{{\cal M}}

\newcommand{\RR}{ \mathbb{R}}

\newcommand{\NN}{ \mathbb{N}}

\newcommand{\EE}{{\mathbb E}}
\newcommand{\TT}{{\mathbb{T}}}

\newcommand{\cL}{{\mathcal L}}

\newcommand{\cP}{{\mathcal P}}

\newcommand{\eet}{{\widetilde \ee}}

\newcommand{\qt}{{\widetilde q}}

\newcommand{\aat}{{\widetilde \aa}}

\newcommand{\mut}{{\widetilde \mu}}

\newcommand{\llt}{{\widetilde \ll}}

\newcommand{\At}{{\widetilde A}}
\newcommand{\Bt}{{\widetilde B}}
\newcommand{\Ct}{{\widetilde C}}

\newcommand{\Ht}{{\widetilde H}}

\newcommand{\Lt}{{\widetilde L}}
\newcommand{\Mt}{{\widetilde M}}

\newcommand{\nt}{{\widetilde n}}

\DeclareMathOperator{\supp}{\mathop{\rm supp}}

\newcommand{\ptt}{{\partial_{t}}}

\newcommand{\pxi}{{\partial_{x_i}}}

\newcommand{\norm}[3][]{{\left\| #2 \right\|_{#3}^{#1}}}
\newcommand{\normc}[2]{{\left\| #1 \right\|_{C^{#2}}}}
\newcommand{\norminf}[2][]{{\left\| #2 \right\|_{\infty}^{#1}}}

\newcommand{\normLp}[3][]{{\left\| #2 \right\|_{L^{#3}}^{#1}}}

\DeclareMathOperator{\divo} {div}

\newcommand{\lp}{\left(}
\newcommand{\rp}{\right)}

\newcommand{\lc}{\left\{}
\newcommand{\rc}{\right\}}
\newcommand{\lb}{\left[}
\newcommand{\rb}{\right]}

\newcommand{\labs}{\left|}
\newcommand{\rabs}{\right|}

\title{On Classical Solutions to the Mean Field Game System of Controls}
\author{Z. Kobeissi\thanks{Laboratoire Jacques-Louis Lions, Univ. Paris Diderot, Sorbonne Paris Cit\'e, UMR 7598, UPMC, CNRS, 75205, Paris, France. zkobeissi@math.univ-paris-diderot.fr}}

\begin{document}
\maketitle

\begin{abstract}
    We consider a class of mean field games
    in which the optimal strategy of  a representative  agent
    depends on  the statistical distribution of the states and controls.
    
    We prove some existence results for the
    forward-backward system of PDEs under rather natural assumptions.
    The main step of the proof consists of obtaining
    a priori estimates on the gradient of the value function by Bernstein's method.
    Uniqueness is also proved under more restrictive assumptions.

    Finally, we discuss some examples to which the
    previously mentioned existence (and possibly uniqueness) results apply.
\end{abstract}

\section{Introduction}
\label{sec:intro}
The theory of Mean Field Games (MFG for short) has been introduced
in the independent works of J.M. Lasry and P.L. Lions
\cite{MR2269875,MR2271747,MR2295621},
and of M.Y. Huang, P.E. Caines
and R.Malham{\'e} \cite{MR2352434,MR2346927}.
It aims at studying   deterministic or stochastic  differential
games (Nash equilibria) as the number of agents tends to infinity.
The agents are supposed to be
rational (given a cost to be minimized, they always choose the optimal strategies), and indistinguishable.
Furthermore, the agents interact via some empirical averages of quantities which depend on the state variable.

At the limit when $N\rightarrow+\infty$, the game may be modeled by  a system of two coupled
partial differential equations (PDEs), which is named the MFG system.
On the one hand, there is
a Fokker-Planck-Kolmogorov
equation describing the evolution
of the statistical distribution $m$ of the state variable;
this equation is a forward in time parabolic equation, and the  initial distribution at time $t=0$
is given.
On the other hand, the optimal value of a generic agent at some time $t$ and state $x$
is noted $u(t,x)$ and is defined as the lowest cost that a representative agent can achieve from time $t$ to $T$ if it is at
state $x$ at time $t$.
The value function satisfies a Hamilton-Jacobi-Bellman equation posed backward in time
with a terminal condition
involving a terminal cost.
In the present work,
we will restrict our attention to
the case when the costs and
the dynamics are periodic in the state variable,
and we will work
in the $d$-dimensional torus $\TT^d$
(as it is often done in the MFG literature for simplicity).
We will take a finite horizon time $T>0$,
and will only consider  second-order non-degenerate MFG systems.
In this case, the MFG system is often written as:
\begin{subequations}\label{eq:MFGtrad}
     \begin{empheq}{align}
            \label{eq:HJBtrad}
            &-\ptt u(t,x)
            - \nu\Delta u(t,x)
            + H(t,x, \nabla_xu(t,x))
            = f(x,m(t))
            &\text{ in }
            (0,T)\times\TT^d,
            \\
            \label{eq:FPKtrad}
            & \ptt m(t,x)
            - \nu\Delta m(t,x)
            -\divo(H_p(t,x,\nabla_xu(t,x))m)
            =0
            &\text{ in }
            (0,T)\times\TT^d,
            \\
            \label{eq:CFutrad}
            &u(T,x)
            =g(x,m(T))
            &\text{ in }
            \TT^d,
            \\
            \label{eq:CImtrad}
            &m(0,x)
            =m_0(x)
            &\text{ in }
            \TT^d.
    \end{empheq}
\end{subequations}
We refer the reader to
\cite{MR3967062}
for some theoretical results on
the convergence of the  $N$-agents Nash equilibrium
 to the solutions of the MFG system.
For a thorough study of the well-posedness
of the MFG system, see the videos
of P.L.Lions' lecture at the
Coll{\`e}ge de France,
and some lecture notes
\cite{Cardaliaguet_notes_on_MFG}.

There is also an important literature on the  probabilistic aspects of MFGs,  see
\cite{MR3072222,MR3332857} for some examples
and \cite{MR3752669,MR3753660} for a detailed presentation
of the probabilistic viewpoint.

For applications of MFGs, numerical simulations are crucial because it is most often impossible to find explicit or semi-explicit solutions to the MFG system.
We refer to \cite{MR3135339} for a survey on finite difference methods and to \cite{MR3821977} for applications to crowd motion.

Most of the literature on MFGs is focused  on the case when the
 mean field interactions only involves the distributions of states.
Here we will consider a more general situation in which
the cost of an individual agent depends on the joint distribution $\mu$
of states and optimal strategies.
To underline this, we choose to use the terminology
{\sl Mean Field Games of Controls (MFGCs)}
for this class of MFGs; the latter terminology  was introduced in \cite{MR3805247}.
Within this framework, the usual MFG system \eqref{eq:MFGtrad}
is replaced by the following MFGC system,
\begin{subequations}\label{eq:MFGC}
     \begin{empheq}{align}
            \label{eq:HJB}
            &-\ptt u(t,x)
            - \nu\Delta u(t,x)
            + H(x, \nabla_xu(t,x),\mu(t))
            = 0 
            &\text{ in }
            (0,T)\times\TT^d,\\
            \label{eq:FPK}
            & \ptt m(t,x)
            - \nu\Delta m(t,x)
            -\divo(H_p(x,\nabla_xu(t,x),\mu(t))m)
            =0
            &\text{ in }
            (0,T)\times\TT^d,\\
            \label{eq:defmu}
            &\mu(t)
            = \Bigl(
            I_d,
            -H_p\lp \cdot,\nabla_xu(t,\cdot),\mu(t)\rp
            \Bigr){\#}m(t)
            &\text{ in } [0,T],\\
            \label{eq:CFu}
            &u(T,x)
            =g(x,m(T))
            &\text{ in }
            \TT^d,\\
            \label{eq:CIm}
            &m(0,x)
            =m_0(x)
            &\text{ in }
            \TT^d.
    \end{empheq}
\end{subequations}
We would like to point out two of
the main difficulties that one may encounter when
studying \eqref{eq:MFGC} and which are not present
in the study of \eqref{eq:MFGtrad}.
\begin{enumerate}[{\bf 1)}]
    \item
        \label{pb:FPmu}
        The joint law of states and controls
        satisfies a fixed point relation
        described by \eqref{eq:defmu}.
    \item
        \label{pb:apriori}
        The HJB equation \eqref{eq:HJB}
        is non-local with respect to $\nabla_xu$.
        Consequently, it is much more difficult
        to obtain uniform a priori estimates
        on $u$ and the its derivatives.
\end{enumerate}
Difficulty \ref{pb:FPmu} is in general not
straightforward and one needs to make assumptions
for the fixed point in $\mu$ to have a unique solution
when $\lp \nabla_xu,m\rp$ are given.
An example in which this fixed point relation
does not admit any solution is given in \cite{achdou2020mean}
Remark $4.3$.

Let us provide a simple illustration 
for describing 
difficulty \ref{pb:apriori}
by comparing the results obtained
when we apply the maximum principle on parabolic equations
to \eqref{eq:HJBtrad} and \eqref{eq:HJB} respectively:
if $u$ satisfies \eqref{eq:HJBtrad} where
$f$ and $g$ are assumed to be 
uniformly bounded with respect to $m$,
then $u$ is uniformly bounded;
under the same assumption on $g$,
if $u$ is a solution to \eqref{eq:HJB}
and $H$ is not uniformly bounded with respect to $\mu$,
we can only say that $u$ is bounded in absolute value
by a constant depending on $\mu$.
The other estimates used in the usual arguments
of existence in MFG sytems suffer the same
lack of uniformity with respect to $\mu$.
Conversely, the estimates of $\mu$ depend
on $\nabla_xu$.
It is not obvious a priori how to combine
the estimates on $\mu$ and $(u,m)$ in order
to obtain uniform estimates on $u$.
Consequently,
compactness results are harder to obtain for \eqref{eq:MFGC}
than for \eqref{eq:MFGtrad}.

The main assumption of this paper,
namely \ref{hypo:Hinvert} and \ref{hypo:Hcontrac}
described below,
is an original structural assumption designed
to address difficulty \ref{pb:FPmu}.
In particular, it implies that the map
\begin{equation*}
    \mu\mapsto\mut=
    \Bigl(I_d,
    -H_p\lp \cdot,\nabla_xu(t,\cdot),\mu\rp
    \Bigr){\#}m,
\end{equation*}
is a contraction in a
convenient metric space,
when $(t,u,m)$ are given.

Moreover, we also assume that the Hamiltonian
$H(x,p,\mu)$ behaves like a power function
when $p$ tends to infinity.
See paragraph \ref{subsec:assum}
for more details.

The main objective of this work is to
discuss existence of the solutions
of the MFGC system \eqref{eq:MFGC}
within this framework.
We will also give a uniqueness result
under a short time horizon assumption.
We refer to \cite{achdou2020mean}
for a numerical application with multiple
solutions.
Indeed, uniqueness does not hold in general
for arbitrary time horizon.
It can be obtained though,
under a monotonicity assumption which is
investigated in the companion paper
\cite{MonoMFGC}.
In \cite{MonoMFGC},
existence and uniqueness
of solutions of the MFGC system
are proved under the above-mentioned
monotonicity assumption and with Hamitonian
having similar growth as in the present paper. 
This monotonicity condition implies that the agents favor
moving in a direction opposite to the mainstream.
Such an assumption
is adapted to some models coming
from finance or economy;
and may be unrealistic in several situations,
in particular in models of crowd motions.
This explains why here we introduce a new
structural assumption and refrain from
assuming monotonicity or investigating
uniqueness in the general case.

\subsection*{Related literature}
In the first articles devoted to MFGCs,
\cite{MR3112690,MR3160525}, D. Gomes and his collaborators have
given several existence results for MFGCs in various cases,
using  the terminology {\sl extended MFGs} instead of {\sl MFGCs}.
For instance, \cite{MR3160525} contains  existence results for
stationary games (infinite horizon)  under the assumption
that some of the parameters involved in the models are small.
We refer to
\cite{MR3941633,MR3805247,MR3325272,MR3752669,MonoMFGC}
for other existence and uniqueness results for MFGC systems.

Uniqueness is a major issue in MFG theory,
it has been proved
for \eqref{eq:MFGtrad}
in \cite{MR2295621,Lions_video}
under an assumptions called
the Lasry-Lions monotonicity on the coupling function $f$
and the terminal cost $g$ in the case of non-local coupling.
This assumption has been extended to MFGC
and discussed in
\cite{MR3112690,MR3752669,MonoMFGC}
in which uniqueness is proved.
It translates the fact that
the agents prefer directions opposite
to the mainstream  direction;
therefore it is not adapted to
a large class of MFGC systems
like crowd motion models
in which an agent
is more likely to go in
the mainstream direction.

The latter example of population dynamic
is the typical application
we had in mind when writing the assumptions
in the present paper,
see paragraphs
\ref{subsec:flocking}
and \ref{subsec:crowd}.
To our knowledge,
existence results for such MFGC systems
have not been discussed in the literature before.
Uniqueness should not hold in general but under
a short-time assumption.
We refer to \cite{achdou2020mean}
in which the MFGC system is discretized
using a finite-difference scheme
and simulations are provided where
the approximating discrete MFGC system
admits several different solutions.

For other applications of MFGCs
we refer to 
\cite{MR3805247} for an model of optimal trading,
\cite{2019arXiv190205461F,MR3359708,MR3755719,
MR2762362,MR4064472} in the
case of competition between firms producing the same goods,
or \cite{MR4054298} for energy storage.

\subsection*{Organization of the paper}
Section \ref{sec:not_ass}
describes the notations,
assumptions
and main results in this paper.
In Section \ref{sec:fixed_point},
we address difficulty
\ref{pb:FPmu} which consists
of inverting the fixed point
relation in $\mu$ \eqref{eq:defmu}
and providing estimates on the resulting
flow of measures.
Section \ref{sec:apriori}
is devoted to proving a priori
estimates on the solutions
to \eqref{eq:MFGC}
and addresses difficulty \ref{pb:apriori}.
Section \ref{sec:ExiUni} contains
the proofs of the main results.
Finally, we discuss several applications
in Section \ref{sec:appli}.
Namely, we study
\begin{itemize}
    \item
        the
Bertrand and Cournot competition
for exhaustible ressources
and introduce an extension
to negatively correlated ressources
(for instance gold and other raw materials);
\item
a model of price impact for high-frequency
trading by Almgren and Chriss
in which we discuss the possibility
for the bid and ask prices to be different;
\item
a first-order flocking model;
\item
    a crowd motion model.
\end{itemize}

\section{Notations and assumptions}
\label{sec:not_ass}
\subsection{Notations and definitions}
\label{subsec:not}
The spaces of probability measures 
are equipped with the weak* topology.
We denote by
$\cP_{\infty}\lp\TT^d\times\RR^d\rp$
the subset of measures $\mu$ in $\cP\lp\TT^d\times\RR^d\rp$
with a second marginal compactly supported.
For $\mu\in\cP_{\infty}\lp\RR^d\times\RR^d\rp$
and $\qt\in[1,\infty)$,
we define the quantities
$\LL_{\qt}(\mu)$ and $\LL_{\infty}(\mu)$ by,
\begin{equation}
    \label{eq:defLq}
\begin{aligned}
    \LL_{\qt}(\mu)
    &=
    \lp\int_{\RR^d\times\RR^d}
    \labs\aa\rabs^{\qt}d\mu\lp x,\alpha\rp\rp^{\frac1{\qt}},
    \\
    \LL_{\infty}(\mu)
    &=
    \sup\lc\labs\aa\rabs,
    (x,\aa)\in\supp\mu\rc.
\end{aligned}
\end{equation}
Jensen inequality states that,
\begin{equation}
    \label{eq:Jensen}
    \LL_{q_1}(\mu)
    \leq
    \LL_{q_2}(\mu),
\end{equation}
for any $1\leq q_1\leq q_2\leq \infty$.

For $R>0$, we denote by
$\cP_{\infty,R}\lp\RR^d\times\RR^d\rp$
the subset of measures $\mu$ in
$\cP_{\infty}\lp\RR^d\times\RR^d\rp$
such that $\LL_{\infty}\lp\mu\rp\leq R$.
The probability measures $\mu$ 
involved in
\eqref{eq:MFGC} and \eqref{eq:MFGCM},
have a particular form,
since they are the images of a measure
$m$ on $\TT^d$
by $\lp I_d,\aa\rp$,
where $\aa$ is a bounded measurable
functions from $\TT^d$ to $\RR^d$;
in particular they are supported on the graph of $\aa$.
For $m\in\cP\lp\TT^d\rp$,
we call
$\cP_m\lp\TT^d\times\RR^d\rp$
the set of such measures.
For $\mu\in\cP_m\lp\TT^d\times\RR^d\rp$,
we set $\aa^{\mu}$ to be the unique 
element of $L^{\infty}\lp m\rp$
such that $\mu=\lp I_d,\aa^{\mu}\rp\#m$.
Here, $\LL_{\qt}(\mu)$ and $\LL_{\infty}(\mu)$
defined in \eqref{eq:defLq} are given by
\begin{equation}
    \label{eq:defLqbis}
\begin{aligned}
    \LL_{\qt}(\mu)
    &=
    \norm{\aa^{\mu}}{L^{\qt}(m)},
    \\
    \LL_{\infty}(\mu)
    &=
    \norm{\aa^{\mu}}{L^{\infty}(m)}.
\end{aligned}
\end{equation}

If $X$ is a normed space and $|\cdot|_X$ is its norm,
for $n\geq 1$ we denote by
$C^0\lp X;\RR^n\rp$ the set of bounded
continuous functions from $X$ to 
$\RR^n$;
it is
endowed 
with the norm
$\norminf{v}=\sup_{x\in X}|v(x)|_X$.

We define $C^{0,1}\lp[0,T]\times\TT^d;\RR\rp$ as the set of the functions
$v\in C^0\lp[0,T]\times\TT^d;\RR\rp$ differentiable at any point with respect to
the state variable, and such that its gradient 
satisfies $\nabla_xv\in
C^{0}\lp[0,T]\times\TT^d;\RR^d\rp$.
This is a Banach space equipped with the norm 
$\norm{v}{C^{0,1}}=\norminf{v}+\norminf{\nabla_xv}$.

For $\bb\in(0,1)$ and $n\geq 1$,
we denote by
$C^{\frac{\bb}2,\bb}\lp[0,T]\times\TT^d;\RR^n\rp$
the parabolic space 
of Hölder continuous functions
which is commonly defined by
\begin{equation*}
    C^{\frac{\bb}2,\bb}\lp[0,T]\times\TT^d;\RR^n\rp 
    = \lc
    \begin{aligned}
        v& \in C^0([0,T]\times\TT^d;\RR^n),
        \exists C>0 \text{ s.t. }
        \forall (t_1,x_1),(t_2,x_2)\in [0,T]\times\TT^d, \\
        &|v(t_1,x_1)-v(t_2,x_2)|\leq
        C\lp|x_1-x_2|^2 +|t_1-t_2|\rp^{\frac{\bb}2} 
    \end{aligned}
    \rc.
\end{equation*}
This is a Banach space equipped with the norm,
\begin{equation*}
    \norm{v}{C^{\frac{\bb}2,\bb}} 
    =
    \norminf{v}
    +\sup_{(t_1,x_1)\neq(t_2,x_2)}
    \frac{|v(t_1,x_1)-v(t_2,x_2)|}
    {\lp|x_1-x_2|^2 +|t_1-t_2|\rp^{\frac{\bb}2}}.
\end{equation*}
The space
$C^{\frac{1+\bb}2,1+\bb}([0,T]\times\TT^d;\RR)$
is defined 
as the set of the functions
$v\in C^{0,1}([0,T]\times\TT^d;\RR)$
such that $\nabla_xv\in C^{\frac{\bb}2,\bb}\lp[0,T]\times\TT^d;\RR^n\rp$
and which admits a finite norm defined by,
\begin{equation*}
    \normc{v}{\frac{1+\bb}2,1+\bb}
    =
    \norminf{v}
    +\norm{\nabla_xv}{C^{\frac{\bb}2,\bb}}
    + \sup_{(t_1,x)\neq (t_2,x)\in [0,T]\times\TT^d}
    \frac{|v(t_1,x)-v(t_2,x)|}{|t_1-t_2|^{\frac{1+\bb}2}}.
\end{equation*}
We set $C^{1,2}\lp[0,T]\times\TT^d;\RR\rp$ to be the set
of functions which admit first derivative with resepct to time
and second derivatives with respect to the state variables,
such that these derivatives are continuous with respect
to time and state.

Throughout the paper,
what we call a solution
to \eqref{eq:MFGC}
is precisely defined
by the following definition.
\begin{definition}
    \label{def:sol}
    The triple $(u,m,\mu)$
    is a solution to \eqref{eq:MFGC}
    if $u\in C^{1,2}([0,T]\times\TT^d)$
    is a pointwise solution
    to the Hamilton-Jacobi-Bellman equation
    \eqref{eq:HJB}
    with terminal condition \eqref{eq:CFu},
    $m\in C^0\lp [0,T]\times\TT^d;\RR\rp$
    is solution to the Fokker-Planck-Kolmogorov equation
    \eqref{eq:FPK}
    in the sense of distribution
    with initial condition \eqref{eq:CIm},
    and $\mu \in C^0\lp [0,T];
    \cP_{\infty}\lp\TT^d\times\RR^d\rp\rp$
    satisfies \eqref{eq:defmu}
    at any $t\in[0,T]$.
\end{definition}
A simple way to overcome difficulty \ref{pb:apriori}
is to assume that the Hamiltonian $H$
and some of its derivatives
admit uniform bounds with respect to $\mu$.
In this case, the
well-posedness of the MFGC system
with a possibly degenerate
diffusion
is investigated in \cite{MR3805247}.
Here we avoid such an assumption for \eqref{eq:MFGC}
but we introduce the following
approximating system
which satisfies it,
\begin{subequations}\label{eq:MFGCM}
     \begin{empheq}{align}
        \label{eq:HJBM}
        &-\ptt u^M(t,x)
        - \nu\Delta u^M(t,x)
        + H(x, \nabla_xu^M(t,x),\mu^M(t))
        =
        0 
        &\text{ in } (0,T)\times\TT^d,\\
        \label{eq:FPKM}
        & \ptt m^M_t(t,x)
        - \nu\Delta m^M(t,x) 
        -\divo(H_p(x,\nabla_xu^M(t,x),\mu^M(t))m^M)
         =
        0
        &\text{ in } (0,T)\times\TT^d,\\
        \label{eq:defmuM}
        &\mu^M(t) 
        =
        \lb I_d,
        T_M\lp -H_p\lp \cdot,\nabla_xu^M(t,\cdot),\mu^M(t)\rp\rp
        \rb{\#}m^M(t)
        &\text{ in } [0,T],\\
        \label{eq:CFuM}
        &u^M(T,x)
        =
        g(x,m^M(T))
        &\text{ in } \TT^d,\\
        \label{eq:CImM}
        &m^M(0)=m_0,
    \end{empheq}
\end{subequations}
    where $M$ is a positive constant and
    $T_M$ is a truncation map defined by
    \begin{equation*}
        T_M(v)=\lc
        \begin{aligned}
            &v \text{ if } |v|\leq M, \\
            &\frac{M}{|v|}v \text{ otherwise.}
        \end{aligned}
        \right.
    \end{equation*}
The latter definition can be naturally extended
to the case when $M=\infty$
by taking $T_{\infty}=Id_{\RR^d}$.
In this case systems \eqref{eq:MFGC} and \eqref{eq:MFGCM}
coincide.
A solution to \eqref{eq:MFGCM}
is defined by replacing \eqref{eq:MFGC}
by \eqref{eq:MFGCM} in Definition \ref{def:sol}.

If $M<\infty$ and $\lp u^M,m^M,\mu^M\rp$
is a solution to \eqref{eq:MFGCM},
$\mu^M(t)$ is compactly supported in
 $B_{\RR^d}(0,M)$ the closed ball in $\RR^d$ centered at $0$
with radius $M$, for any $t\in[0,T]$.
Consequently, $u^M$ and $m^M$ should satisfy
estimates depending on $M$ and uniform with respect to $\mu^M$.
Therefore, compactness results for \eqref{eq:MFGCM}
should be less demanding than for \eqref{eq:MFGC}
and difficulty \ref{pb:apriori} should vanish.

\subsection{Assumptions}
\label{subsec:assum}
Let us start with some reasonable assumptions about the regularity and the boundedness
of the Hamiltonian, the terminal cost and the inital distribution of agents.
We introduce two constants: $C_0>0$ and $\bb_0\in(0,1)$.
    \begin{enumerate}[label=\bf{A\arabic*}]
            \item \label{hypo:Hp}
                $H=H(x,p,\mu)$ is convex with respect to $p$,
                and differentiable with respect to $(x,p)$;
                $H_p$ is locally $\bb_0$-Hölder
                continuous
                with respect to $p$;
                $H$ and $H_p$ are continuous
                with respect to $\mu$ on
                $\cP_{\infty,R}\lp\TT^d\times\RR^d\rp$
                for any $R>0$,
                where
                $\cP_{\infty,R}\lp\TT^d\times\RR^d\rp$
                is defined in paragraph \ref{subsec:not}
                and equipped with the weak* topology.
        
            \item\label{hypo:gregx} 
                $g:\TT^d\times\cP\lp \TT^d\rp\rightarrow \RR$
                is  continuous,
                and we suppose that
                $x\mapsto g(x,m)$ is in
                $C^{2+\bb_0}\lp \TT^d\rp$,
                with a norm bounded
                uniformly with respect to $m$, i.e.
                \begin{equation*}
                    \normc{g(\cdot,m)}{2+\bb_0}
                    \leq
                    C_0,
                    \quad \forall m\in\cP(\TT^d).
                \end{equation*}
            
            \item \label{hypo:mreg} 
        $m_0\in \cP(\TT^d)$
        is absolutely continuous
        with respect to the
        Lebesgue measure on $\TT^d$
        and we also name $m_0$
        its density (abuse of notation).
        Assume that $m_0\in C^{\bb_0}(\TT^d)$
        and is positive (see Remark \ref{rem:relaxm0}
        below to drop out the positivity assumption).
    \end{enumerate}
These assumptions are not restrictive when looking for
solutions with the regularity given
in Definition \ref{def:sol}.
However, they can be relaxed if we are interested in
weaker solutions of systems \eqref{eq:MFGC}.

In this paper we consider nonlocal coupling through the controls.
More precisely, we assume that these interactions involve
the quantity $\LL_{q_0}(\mu)$ defined in \eqref{eq:defLq}.

Let us introduce the assumptions used
to address difficulty \ref{pb:FPmu}
which consists of
solving the fixed point relations in $\mu$
given in 
\eqref{eq:defmu} and \eqref{eq:defmuM},
when $(u,m)$ are fixed and have the same
regularity as in Definition \ref{def:sol}.
We introduce $\ll_0\in[0,1)$,
for all $\lp x,p,m\rp\in\TT^d\times\RR^d\times\cP\lp\TT^d\rp$,
and $\mu,\mu^1,\mu^2\in\cP_m\lp\TT^d\times\RR^d\rp$,
we assume that,
\begin{enumerate}[label=\bf{FP\arabic*}]
    \item \label{hypo:Hinvert}
        $|H_p(x,p,\mu)|
        \leq C_0(1+|p|^{q-1})
        +\ll_0 \LL_{q_0}(\mu)$.
    \item \label{hypo:Hcontrac}
        $\labs H_p(x,p,\mu^1)
        -H_p(x,p,\mu^2)\rabs
        \leq
        \ll_0
        \norm{\aa^{\mu^1}
        -\aa^{\mu^2}}{L^{q_0}(m)}$.
\end{enumerate}
These structural assumptions for MFGC are new in the literature 
and participate to the originality and novelty of the results
presented in this paper.
Moreover they do not seem to be restrictive as it is explained
in what follows.

We recall that the optimal control of a representative agent
is given by $\aa=-H_p\lp x,\nabla_xu,\mu\rp$.
Since $\LL_{q_0}(\mu)$ is homogeneous to the norm of a control,
we cannot expect the dependency of $H_p$ upon $\mu$
to involve an exponent larger than one.
Moreover if $m$ is the first marginal of $\mu$,
taking the $L^{q_0}(m)$-norm in \ref{hypo:Hinvert}
makes $\LL_{q_0}(\mu)$ appear in both sides
of the resulting inequality;
this explains
the form of the right-hand side in \ref{hypo:Hinvert}
and the necessity
of choosing $\ll_0$ smaller than $1$.
Similar arguments can provide insights on \ref{hypo:Hcontrac},
by noticing that if $\LL_{q_0}$
was seen as a norm on $\cP_m\lp\TT^d\times\RR^d\rp$
then $\norm{\aa^{\mu^1}-\aa^{\mu^2}}{L^{q_0}(m)}$
would be the associated distance.
We refer to Remark $4.3$ in \cite{achdou2020mean}
for a concrete example of a MFGC system which does not admit
solution if $\ll_0=1$.

As in a large part of the literature on MFG or HJB equations, we
consider Hamiltonians that are power-like functions in $p$ at least
asymptotically. Let $q\in(1,\infty)$ be this asymptotic exponent,
and $q'$ the conjugate exponent of $q$
defined by $q'=\frac{q}{q-1}$.
Namely, we assume that $H$ satisfies the
following inequalities,
for all $x\in\TT^d$, $p\in\RR^d$,
$m\in\cP\lp\TT^d\rp$,
and $\mu\in\cP_{m}\lp \TT^d\times \RR^d\rp$,
\begin{enumerate}[label=\bf{B\arabic*}]
    \item \label{hypo:Hh}
        $|H(x,0,\mu)|
        \leq C_0
        +\ll_2\LL_{q_0}(\mu)^{q'}$,
        with $\ll_2\geq0$.
    \item \label{hypo:Hx}
        $|H_x(x,p,\mu)|
        \leq C_0\lp1+|p|^{q}+\LL_{q_0}(\mu)^{q'}\rp$.
    \item \label{hypo:Hpower}
        $H_p(x,p,\mu)\cdot p -H(x,p,\mu)
        \geq 
        C_0^{-1}\lp|p|^q
        -\ll_1\LL_{q_0}\lp\mu\rp^{q'}\rp
        -C_0$,
        where $\ll_1$ is a nonnegative constant
        satisfying
        $0\leq\ll_1<\frac{(1-\ll_0)^{q'}}{C_0^{q'}}$.
\end{enumerate}
One may notice that the dependencies of $H$ upon $p$ and $\mu$
involve different exponents (which happen to be equal when $q=2$).
Indeed the Legendre transform applied to a power-like function
make the exponent change into its conjugate.
Since $H$ is defined as the Legendre transform of the Lagrangian $L$,
the exponent in the dependency of $L$ upon $\aa$ should be $q'$.
Moreover, $\LL_{q_0}(\mu)$ is homogeneous to the norm of a control,
therefore $L$ should 
at most involve $\Lambda_{q_0}(\mu)^{q'}$.
Going back to the Hamiltonian by the Legendre transform,
the exponent on $\LL_{q_0}(\mu)$ stays the same
which explains the right-hand side
in \ref{hypo:Hh}-\ref{hypo:Hpower}.
One may find the abovementionned growth conditions on $L$
in \cite{MonoMFGC}.

Assumption \ref{hypo:Hpower}
is a convexity property
of $H$ with respect to $p$.
In MFG without coupling through the controls, such an assumption
is common, the only difference is that the term in
$\LL_{q_0}(\mu)$ does not appear.
This assumption will be particularly useful to obtain
energy integral estimates by taking advantage
of the duality properties
of the forward-backward systems
\eqref{eq:HJB}, \eqref{eq:FPK}
and \eqref{eq:HJBM}, \eqref{eq:FPKM}.
The inequality satisfied by $\ll_1$
is needed in the calculation for getting these estimates.
Let us mention that the right-hand side in this inequality
comes from the estimates in Lemma \ref{lem:FPV},
and that the constant $C_0$ can be identified with the
one in \ref{hypo:Hinvert}.


In order to obtain classical solutions 
of the HJB equations \eqref{eq:HJB} and \eqref{eq:HJBM},
we need
Hölder continuity
of $(t,x)\mapsto H\lp x,\nabla_xu(t,x),\mu(t)\rp$.
While the space regularity of the latter map
is straightforward here,
its time regularity may be more demanding
and we need assumptions which allow one
to compare $H$ at different
measures $\mu^1,\mu^2\in\cP_{\infty}\lp\TT^d\times\RR^d\rp$.
Assumption \ref{hypo:Hcontrac} is not enough since
it requires $\mu^1$ and $\mu^2$
to share the same marginal with respect to $\TT^d$.
\begin{enumerate}[label=\bf{T}]
   \item
        \label{hypo:Hlip}
        For $R>0$, there exists
        a constant $C_R>0$ 
        such that
        \begin{align*}
            \labs H\lp x,p,\mu^1\rp
            -H\lp x,p,\mu^2\rp\rabs
            &\leq
            C_R\lp
            \norminf[\bb_0]{m^1-m^2}
            +\norminf{\aa^{\mu^1}-\aa^{\mu^2}}
            \rp,
            \\
            \labs H_p\lp x,p,\mu^1\rp
            -H_p\lp x,p,\mu^2\rp\rabs
            &\leq
            C_R\lp
            \norminf[\bb_0]{m^1-m^2}
            +\norminf{\aa^{\mu^1}-\aa^{\mu^2}}
            \rp, 
        \end{align*}
        for $\lp x,p,m^i,\mu^i\rp$
        such that $\lp x,p\rp\in\TT^d\times\RR^d$
        with $|p|\leq R$,
        $m^i\in\cP\lp\TT^d\rp\cap C^0\lp\TT^d\rp$
        with $m^i\geq R^{-1}$,
        $\mu^i\in\cP_{m^i}\lp\TT^d\times\RR^d\rp$
        with $\aa^{\mu^i}\in C^0\lp\TT^d\times\RR^d\rp$
        and $\norminf{\aa^{\mu^i}}\leq R$,
        $i=1,2$.
\end{enumerate}
One may notice that when $\mu^1$ and $\mu^2$
have the same first marginal with respect to $\TT^d$
the second inequality in \ref{hypo:Hlip} is implied
by \ref{hypo:Hcontrac}.
If one is only interested in weak solution
to \eqref{eq:MFGC},
\ref{hypo:Hlip} can be removed.

\begin{remark}
    \label{rem:relaxm0}
    Letting $C_R$ depends on $\norminf{\lp m^i\rp^{-1}}$
    was motivated by models of population dynamics
    which are discussed in paragraphs
    \ref{subsec:flocking} and \ref{subsec:crowd}.
    The drawback of this assumption is that we have
    to assume that the initial distribution
    of agents $m_0$ is positive.

    All the results in this paper hold
    if we do not assume $m_0$
    to be positive in \ref{hypo:mreg},
    and we remove the condition
    $m^i\geq R^{-1}$ in \ref{hypo:Hlip}.
\end{remark}

\subsection{Main results}
\label{subsec:MFGCMain}
We recall that assumptions \ref{hypo:Hinvert}
and \ref{hypo:Hcontrac}
are designed to address difficulty \ref{pb:FPmu},
and \ref{hypo:Hlip} to obtain time regularity
of the fixed point $\mu$ in \ref{eq:defmu}
or \ref{eq:defmuM}.
More precisely,
we state
the following lemma
that will be proved in
Section \ref{sec:fixed_point}.
\begin{lemma}
    \label{lem:FPmu}
    Assume \ref{hypo:Hp},
    \ref{hypo:Hinvert}, \ref{hypo:Hcontrac}
    and \ref{hypo:Hlip}.
    Take $p\in C^{\frac{\bb}2,\bb}\lp[0,T]\times\TT^d;\RR^d\rp$
    and $m\in C^{\frac{\bb}2,\bb}\lp[0,T]\times\TT^d\rp$
    such that $m\geq R^{-1}$
    and $m(t)\in\cP\lp\TT^d\rp$
    for $t\in[0,T]$,
    where $\bb\in(0,1)$ and $R>0$ are constants.
    For any $t\in[0,T]$, there exists a unique
    $\mu(t)\in\cP\lp \TT^d\times\RR^d\rp$
    satisfying
    \begin{equation*}
        \mu^M(t)
        =
        \lb I_d,
        T_M\lp -H_p\lp \cdot,p(t,\cdot),\mu^M(t)\rp\rp
        \rb{\#}m(t),
    \end{equation*}
    where $M\in(0,\infty]$.
    Moreover,
    the map $(t,x)\mapsto\aa^{\mu^M(t)}(x)$
    is in 
    $C^{\frac{\bb\bb_0}2,\bb\bb_0}\lp[0,T]\times\TT^d;\RR^d\rp$,
    and its associated norm can be estimated from above
    by a constant which depends on
    $\norm{p}{C^{\frac{\bb}2,\bb}}$,
    $\norm{m}{C^{\frac{\bb}2,\bb}}$,
    and the constants in the assumptions.
\end{lemma}
In Section \ref{sec:apriori},
we prove the a priori estimates stated in the
following lemma.
\begin{lemma}
    \label{lem:MFGCapriori}
    Assume
    \ref{hypo:Hp}-\ref{hypo:mreg},
    \ref{hypo:Hx},
    \ref{hypo:Hpower},
    \ref{hypo:Hinvert},
    \ref{hypo:Hcontrac}
    and
    \ref{hypo:Hlip}.
    If $(u,m,\mu)$ is
    a solution to \eqref{eq:MFGCM}
    for $M\in(0,\infty]$,
    then
    \begin{itemize}
        \item
            $\norminf{\nabla_xu}\leq C\lp 1+\norminf{u}\rp$
            and
            $\norminf{u}\leq C\lp 1+\norminf[q]{\nabla_xu}\rp$,
            where $C$ is independent of $M$ and depends only on the
            constants in the assumptions,
        \item
            $m$ is positive,
        \item
            $u\in
            C^{1+\frac{\bb}{2},2+\bb}
            \lp [0,T]\times\TT^d\rp$,
        \item
            $m\in C^{\frac{\bb}2,\bb}\lp[0,T]\times\TT^d;\RR\rp$,
        \item
            $(t,x)\mapsto\aa^{\mu(t)}(x)$
            is in 
            $C^{\frac{\bb}2,\bb}\lp[0,T]\times\TT^d;\RR^d\rp$,
    \end{itemize}
    where $\bb\in\lp0,\bb_0^2\rp$.
    Moreover, $\norminf{m^{-1}}$ and the norms associated with the last three
    items above depend only on $\norminf{u}$,
    $\norminf{m_0^{-1}}$,
    $\bb$ and the constants in the assumptions.
\end{lemma}
These estimates are weaker than
their equivalents for MFG systems without
interaction through controls.
In particular,
$u$ is not uniformly
bounded in $\norminf{\cdot}$-norm.
However, we believe that our estimate
of $\norminf{\nabla_xu}$ is the best that
we can achieve in our framework
since its right-hand side
should be at least linear
with respect to $\norminf{u}$.
To our knowledge,
such an estimate for systems of MFG
with nonlocal dependency on $\nabla_xu$
(or more generally for MFG systems
in which we do not have a uniform
a priori estimate on $u$)
is new in the literature.

Here, these a priori estimates
are not sufficient to address the
difficulty \ref{pb:apriori} and to
obtain existence of solutions.
However,
existence can be obtained under
several different kinds of assumptions;
below, we supply a list of existence
results under various assumptions:
\begin{theorem}
    \label{thm:MFGCexi}
    Assume
    \ref{hypo:Hp}-\ref{hypo:mreg},
    \ref{hypo:Hh}-\ref{hypo:Hpower},
    \ref{hypo:Hinvert},
    \ref{hypo:Hcontrac},
    \ref{hypo:Hlip}.
    There exists a solution
    to \eqref{eq:MFGC}
    if one of the following
    assertions is satisfied
    \begin{enumerate}[a)]
        \item
            \label{exisub:H0q}
            $q_0\leq q'$ and
            $\labs H(x,0,\mu)\rabs
            \leq
            C_0\lp 1 +  \LL_{q_0}\lp\mu\rp^{\qt}\rp$,
            where $\qt$ is a constant satisfying $\qt< q'$
            (Proposition \ref{thm:ExiExpo}),
        \item
            \label{exisub:ll}
            $q_0\leq q'$
            and $\ll_1+C_0\ll_2<\frac{(1-\ll_0)^{q'}}{C_0^{q'}}$
            where $\ll_1$ and $\ll_2$ are respectively defined
            in \ref{hypo:Hpower} and \ref{hypo:Hh},
            the $C_0$ on the left-hand side comes
            from $C_0^{-1}$ in \ref{hypo:Hpower}
            and the $C_0$ on the right-hand side
            comes from \ref{hypo:Hinvert}
            (Proposition \ref{thm:ExiLowParam}),
        \item
            \label{exisub:H0}
            $\labs H(x,0,\mu)\rabs
            \leq
            C_0\lp 1+\LL_{q_0}\lp\mu\rp^{q'-1}\rp$,
            for any $\lp x,\mu\rp\in\TT^d\times\cP\lp\TT^d\times\RR^d\rp$
            (Proposition \ref{thm:ExiBound}),
        \item
            \label{exisub:Hx}
            $\labs H_x (x,p,\mu)\rabs
            \leq
            C_0\lp 1+|p|
            +\LL_{q_0}\lp\mu\rp^{q'-1}\rp$,
            for any $\lp x,p,\mu\rp\in\TT^d\times\RR^d\times\cP\lp\TT^d\times\RR^d\rp$
            (Proposition \ref{thm:ExiLowSpaDep}),
        \item
            \label{exisub:T}
            $T\leq T_0$,
            where $T_0$ is a constant depending on the constants in the assumptions
            (Proposition \ref{thm:ExiLowT}).
    \end{enumerate}
\end{theorem}
An other additional assumption under which existence holds
is the monotonicity condition addressed in \cite{MonoMFGC}.

We also give a uniqueness result
under a short time horizon assumption.
\begin{theorem}[\emph{Uniqueness with short time horizon}]
    \label{thm:UniLowT}
    Assume
    \ref{hypo:Hp}, 
    \ref{hypo:gregx},
    \ref{hypo:mreg},
    \ref{hypo:Hh},
    \ref{hypo:Hx},
    \ref{hypo:Hpower},
    \ref{hypo:Hinvert},
    \ref{hypo:Hcontrac},
    and that the following three assumptions
    are satisfied,
    \begin{itemize}
        \item
            $H_p$ is locally
            Lipschitz continuous with respect to $p$,
        \item
            $g$ satisfies
    \begin{equation}
        \label{eq:gregm}
        \norm{g(\cdot,m^1)-g(\cdot,m^2)}{C^{1+\bb}}{}
        \leq
        C_0W_{q_1}\lp m^1,m^2\rp,
    \end{equation}
    for any $m^1,m^2\in\cP\lp\TT^d\rp$,
    where $q_1\in[1,\infty)$ and
    $W_{q_1}$ is the $q_1$-Wassertein distance
    on measures,
        \item
            the two inequalities in
            \ref{hypo:Hlip}
            hold when we replace
            $\norminf[\bb_0]{m^1-m^2}$ by
            $W_{q_1}\lp m^1,m^2\rp$.
    \end{itemize}
    There exists $T_1>0$ such that
    if $T<T_1$
    then there is at most
    one solution to \eqref{eq:MFGC}.
\end{theorem}
We believe that this uniqueness result
can be easily extended to more general Hamiltonians,
but that the short-time assumption is essential.
Indeed numerical examples in which non-uniqueness occurs
are presented in \cite{achdou2020mean}.
In these examples, we consider groups of agents
who start from some crowded areas at time $t=0$,
and travel through the domain to arrive at
some target areas.
Imposing a short time assumption in such an example
results in the agents not trying to reach the targets at all.
Indeed in this case the kinetic cost makes it more expensive
for them to cross the domain very quickly before the end
of the game than to do nothing
and just wait passively at their starting point.
For this reason we were not interested in finding less
restrictive assumptions in Theorem \ref{thm:UniLowT}.
This theorem should be only seen as an example of
uniqueness result with a short time horizon assumption.
In particular we wanted the proof
in paragraph \ref{subsec:shorttime} to stay simple.

\begin{remark}
    \begin{enumerate}[\roman*)]
        \item
            In this work, we only consider MFGC systems
            in the $d$-dimensional torus $\TT^d$.
            However, we believe that our existence
            results (Theorem \ref{thm:MFGCexi}) hold
            under the same assumptions on the Euclidean space $\RR^d$,
            and that the method introduced in \cite{MonoMFGC}
            to pass from $\TT^d$ to $\RR^d$ can applied here.
        \item
    We did not include the case $q=1$ in this work
    (i.e. when the Hamiltonian is Lipschitz continuous in $p$).
    In this case,
    systems \eqref{eq:MFGC} and \eqref{eq:MFGCM}
    coincide when $M$ is large enough,
    therefore there exists a solution
    to \eqref{eq:MFGC} under assumptions
    \ref{hypo:Hp}-\ref{hypo:mreg},
    \ref{hypo:Hh}-\ref{hypo:Hpower},
    \ref{hypo:Hinvert},
    \ref{hypo:Hcontrac} and 
    \ref{hypo:Hlip},
    by the same arguments as in
    Lemma \ref{lem:existM}.
    \end{enumerate}
\end{remark}

\section{The fixed point relation in $\mu$
and the proof of Lemma \ref{lem:FPmu}}
\label{sec:fixed_point}
We recall that \eqref{eq:MFGC} and
\eqref{eq:MFGCM} conincide when $M=\infty$.
Here, we take $M\in(0,\infty]$.

The following lemma takes advantage of the structural assumptions
\ref{hypo:Hinvert} and \ref{hypo:Hcontrac}
to solve the fixed point relations
\eqref{eq:defmu} and \eqref{eq:defmuM}
which consists of difficulty \ref{pb:FPmu}.
It also states a priori estimates
on $\mu$ which will be of great use
in the next section to obtain
a priori estimates on $u$ and its derivatives.
\begin{lemma}
\label{lem:FPV}
    Assume \ref{hypo:Hp},
    \ref{hypo:Hinvert} and \ref{hypo:Hcontrac}.
    Take $p\in C^0\lp \TT^d;\RR^d\rp$,
    and $m\in\cP(\TT^d)$.
    The following two assertions are satisfied.
    \begin{enumerate}
        \item[(i)] There exists a unique
            $\mu^M \in \cP(\TT^d\times\RR^d)$ such that 
        \begin{equation}
        \label{eq:muFP}
            \mu^M
            =
            \lb I_d,
            T_M\lp -H_p\lp \cdot,p(\cdot),\mu^M\rp\rp
            \rb{\#}m.
        \end{equation}
        For any $\qt\in[1,\infty]$,
        it satisfies
        \begin{equation}
        \label{eq:boundL}
            \LL_{\qt}\lp\mu^M\rp
            \leq
            \frac{C_0}{1-\ll_0}
            \lp 1 +\norm{\labs p\rabs^{q-1}}{L^{\max\lp q_0,\qt\rp}(m)}\rp.
        \end{equation}

        \item[(ii)]The map 
        $(p,m) \mapsto \mu^M$ given by \eqref{eq:muFP},
        is continuous from 
            $C^0\lp \TT^d;\RR^d\rp\times \cP(\TT^d)$
            to $\cP(\TT^d\times\RR^d)$.
            We recall that
            the spaces of measures
            are equipped with the weak-* topology.
\end{enumerate}
\end{lemma}

\begin{proof}
\begin{enumerate}
    \item[$(i)$]
    Let us define the following map,
    \begin{equation*}
        \Phi_{(p,m)}^M:
        \qquad
        \begin{aligned}
            C^0&\lp\TT^d;\RR^d\rp
            \rightarrow
            C^0\lp\TT^d;\RR^d\rp
            \\
            &\aa
            \mapsto
            \lc
            \begin{aligned}
                &\TT^d
                \rightarrow
                \RR^d
                \\
                &x
                \mapsto
                T_M\lp -H_p\lp x,p(x),\lp I_d,\aa\rp\# m\rp\rp.
            \end{aligned}
        \right.
        \end{aligned}
    \end{equation*}
    This map is well defined by \eqref{hypo:Hp}.
    It is $\ll_0$-Lipschitz continuous
    by \ref{hypo:Hcontrac} and the fact that $T_M$
    is $1$-Lipschitz continuous,
    we recall that $\ll_0<1$.
    Therefore it admits a unique fixed point by the Banach
    fixed point theorem.
    If $\mu^M\in\cP\lp\TT^d\times\RR^d\rp$ satisfies
    \eqref{eq:muFP} then $\aa^{\mu^M}$ is the only fixed point
    of $\Phi_{(p,m)}^M$.
    Conversely, if we denote by $\aa$ the fixed point of $\Phi_{(p,m)}^M$,
    then $\mu^M$ defined by $\mu^M=\lp I_d,\aa\rp\#m$ satisfies \eqref{eq:muFP}.
    This implies that \eqref{eq:muFP} admits a unique fixed point that we name
    $\mu^M$ in what follows.
    From \ref{hypo:Hinvert},
    $\mu^M$ satisfies,
    \begin{align*}
        \LL_{\qt}\lp\mu^M\rp
        &=
        \norm{\aa^{\mu^M}}{L^{\qt}(m)}
        \\
        &\leq
        \norm{C_0\lp1+\labs p\rabs^{q-1}\rp
        +\ll_0 \LL_{q_0}\lp\mu^M\rp}{L^{\qt}(m)}
        \\
        &\leq
        C_0\lp 1 +\norm{\labs p\rabs^{q-1}}{L^{\qt}(m)}\rp
        +\ll_0 \LL_{\qt}\lp\mu^M\rp,
    \end{align*}
    for $\qt\geq q_0$,
    where we obtained the last line by using the triangle inequality
    for the $L^{\qt}$-norm, and \eqref{eq:Jensen}.
    This implies \eqref{eq:boundL} for any $\qt\geq q_0$.
    Then we extend this result to $1\leq \qt<q_0$ by combining
    \eqref{eq:Jensen} and \eqref{eq:boundL} applied to $q_0$.

    \item[$(ii)$]
        Let $(p^n,m^n)_{n\in\NN}
        \in \lp C^0\lp \TT^d\times\RR^d\rp
        ;\cP(\TT^d)\rp^{\NN}$
        be a convergent sequence to $(p,m)$
        in $C^0\lp \TT^d;\RR^d\rp\times\cP\lp \TT^d\rp$.
        We define
        $\mu^N$ as before,
        and $\lp\mu^{N,n}\rp_{n\in\NN}$
        the fixed points
        satisfying
        \begin{equation}
            \label{eq:muFPn}
            \mu^{N,n}
            =
            \lb I_d,
            T_M\lp
            -H_p\lp\cdot ,p^n(\cdot),\mu^{N,n}\rp\rp\rb{\#}m^n
        \end{equation}
        for $n\in\NN$.
        The sequence $\lp p^n\rp_{n\in\NN}$
        is bounded in $C^0\lp \TT^d;\RR^d\rp$,
        thus \eqref{eq:boundL} with $\qt=\infty$
        yields that
        $\lp\mu^{N,n}\rp_{n\in\NN}$
        are uniformly compactly supported.
        The sequence $\lp\mu^{N,n}\rp$ is
        compact in
        $\cP\lp \TT^d\times\RR^d\rp$
        endowed with the weak-* topology. 
        Let $\mut$ be the limit
        of a subsequence $\lp\mu^{N,\vp(n)}\rp_{n\in\NN}$,
        for $\vp:\NN\to\NN$ an increasing function.
        By continuity of $H_p$ and $T_M$,
        we can pass to the limit in \eqref{eq:muFPn}
        taken at $\vp(n)$
        when $n$ tends to infinity,
        this gives that $\mut$ satisfies
        the same fixed point relation as $\mu$. 
        By uniqueness of this fixed point,
        we deduce that $\mut=\mu$. 
        This implies that the entire sequence
        $\lp\mu^{N,n}\rp$ tends to $\mu$.

        Therefore the map
        $(p,m)\mapsto \mu^M$
        is continuous from
        $C^0\lp \TT^d;\RR^d\rp
        \times \cP\lp \TT^d\rp$
        to $\cP\lp \TT^d\times\RR^d\rp$.
\end{enumerate}
\end{proof}
In particular if $q_0\leq q'$,
\eqref{eq:Jensen} and 
\eqref{eq:boundL} yield
\begin{equation}
    \label{eq:boundLq0}
    \LL_{q_0}(\mu)
    \leq
    \LL_{q'}(\mu)
    \leq
    \frac{C_0}{1-\ll_0}
    \lp 1+\norm[q-1]{p}{L^{q}(m)}\rp,
\end{equation}
and then we use the inequality
$(a+b)^{q'}\leq
\frac{a^{q'}}{\th^{q'-1}}
+\frac{b^{q'}}{(1-\th)^{q'-1}}$
which holds for $a,b>0$
and for any $\th\in(0,1)$,
to obtain
\begin{equation}
    \label{eq:boundLq0q'int}
    \LL_{q_0}(\mu)^{q'}
    \leq
    \frac{C_0^{q'}}{(1-\ll_0)^{q'}}
    \lp \th^{1-q'}
    +(1-\th)^{1-q'}\norm[q]{p}{L^{q}(m)}\rp.
\end{equation}
If $q\in[1,\infty]$ without
restriction,
we obtain
\begin{equation}
    \label{eq:boundLq0q'inf}
    \LL_{q_0}(\mu)^{q'}
    \leq
    \frac{C_0^{q'}}{(1-\ll_0)^{q'}}
    \lp \th^{1-q'}
    +(1-\th)^{1-q'}\norminf[q]{p}\rp.
\end{equation}
These latter three inequalities will be of great use in Section
\ref{sec:apriori} for getting a priori estimates.

Given $(u,m)$
as regular as in definition \ref{def:sol},
we can use Lemma \ref{lem:FPV} to prove that
the fixed point relations \eqref{eq:defmu} and \eqref{eq:defmuM}
are well-posed,
and that if $(u,m,\mu)$ is a solution to \eqref{eq:MFGC}
or \eqref{eq:MFGCM} then $\mu$ is continuous with respect
to time.
However,
we need a better regularity in time
to get classical solution of the HJB equations \eqref{eq:HJB}
and \eqref{eq:HJBM}.
In Lemma \ref{lem:distmu}, we use \ref{hypo:Hlip} to obtain
an estimate of the distance between two fixed points
of \eqref{eq:muFP} associated with different $(u,m)$.
We will be particularly interested in using this estimate
on a solution to \eqref{eq:MFGCM} at different times.

\begin{lemma}
    \label{lem:distmu}
    Assume \ref{hypo:Hp},
    \ref{hypo:Hinvert}, \ref{hypo:Hcontrac}
    and \ref{hypo:Hlip}.
    Take $p^1, p^2
    \in C^0\lp \TT^d;\RR^d\rp$,
    and $m^1,m^2\in \cP\lp \TT^d\rp\cap C^0\lp\TT^d;\RR\rp$
    some positive probability measures.
    We define $\mu^1, \mu^2\in \cP\lp \TT^d\times\RR^d\rp$ 
    as the fixed point in $(i)$ in Lemma \ref{lem:FPV}
    associated with $\lp p^1,m^1\rp$ and $\lp p^2,m^2\rp$, respectively.
    There exists a constant $C$ such that
    \begin{equation}
        \label{eq:distmu}
        \norminf{\aa^{\mu^1}-\aa^{\mu^2}}
        \leq
        C\lp \norminf[\bb_0]{p^1-p^2}
        +\norminf[\bb_0]{m^1-m^2}\rp,
    \end{equation}
    where $C$ depends on $\norminf{p^i}$,
    $\norminf{\lp m^i\rp^{-1}}$,
    for $i=1,2$,
    and the constants in the assumptions.
\end{lemma}

\begin{proof}
    We define $\mut$
    by $\mut=\lp I_d,\aa^{\mu^1}\rp\#m^2$.
    The triangle inequality
    and the fact that $T_M$ is a contraction
    imply that for any $x\in\TT^d$,
    \begin{align*}
        \labs \aa^{\mu^1}(x)
        - \aa^{\mu^2}(x)\rabs
        &\leq
        \labs
        H_p\lp x,p^1(x),\mu^1\rp
        -H_p\lp x,p^2(x),\mu^2\rp
        \rabs
        \\
        &\!\begin{multlined}[t][10.5cm]
            \leq
            \labs
            H_p\lp x,p^1(x),\mu^1\rp
            -H_p\lp x,p^1(x),\mut\rp
            \rabs
            \\
            +\labs
            H_p\lp x,p^1(x),\mut\rp
            -H_p\lp x,p^1(x),\mu^2\rp
            \rabs
            \\
            +\labs
            H_p\lp x,p^1(x),\mu^2\rp
            -H_p\lp x,p^2(x),\mu^2\rp
            \rabs.
        \end{multlined}
    \end{align*}
    The measures $\mu^1$ and $\mut$
    are the image measures
    by the same function
    $\lp I_d,\aa^{\mu^1}\rp$,
    of $m^1$ and $m^2$ respectively.
    From 
    \ref{hypo:Hlip},
    we obtain
    \begin{equation*}
        \labs
        H_p\lp x,p^1(x),\mu^1\rp
        -H_p\lp x,p^1(x),\mut\rp
        \rabs
        \leq
        C_R\norminf[\bb_0]{m^1-m^2},
    \end{equation*}
    where $R=\max\lp\norminf{p^i},\norminf{\lp m^i\rp^{-1}}\rp$
    and $C_R$ is the constant defined in \ref{hypo:Hlip}.
    We recall that $\LL_{\infty}\lp\mu^i\rp$
    can be estimated from above by a quantity
    which only depends on $\norminf{p^i}$ and
    the constants in the assumptions,
    by \eqref{eq:boundL}.

    Since $\mut$ and $\mu^2$ have the same marginal
    with respect to $\TT^d$,
    \ref{hypo:Hcontrac} yields that,
    \begin{equation*}
        \labs
        H_p\lp x,p^1(x),\mut\rp
        -H_p\lp x,p^1(x),\mu^2\rp
        \rabs
        \leq
        \ll_0
        \norminf{\aa^{\mu^1}-\aa^{\mu^2}}.
    \end{equation*}
    Then $H_p$ is locally $\bb_0$-Hölder continuous
    by \ref{hypo:Hp} so,
    \begin{equation*}
        \labs
        H_p\lp x,p^1(x),\mu^2\rp
        -H_p\lp x,p^2(x),\mu^2\rp
        \rabs
        \leq
        C\norminf[\bb_0]{p^1-p^2},
    \end{equation*}
    for some constant $C$.
    Combining the latter four inequalities,
    we obtain,
    \begin{equation*}
        \norminf{\aa^{\mu^1}-\aa^{\mu^2}}
        \leq
        C_1\lp\norminf[\bb_0]{p^1-p^2}
        +\norminf[\bb_0]{m^1-m^2}\rp
        +\ll_0\norminf{\aa^{\mu^1}-\aa^{\mu^2}},
    \end{equation*}
    which implies \eqref{eq:distmu}
    up to replacing $C$ with
    $\lp 1-\ll_0\rp^{-1}\max\lp C,C_R\rp$.
\end{proof}
Lemma \ref{lem:FPmu}
is a straightfoward consequence
of Lemmas \ref{lem:FPV} and \ref{lem:distmu}.

\section{A priori estimates
and the proof of Lemma \ref{lem:MFGCapriori}}
\label{sec:apriori}
Here we take $M\in(0,\infty]$,
and $(u,m,\mu)$
a solution to \eqref{eq:MFGCM}
defined in Definition \ref{def:sol}.
We will look for estimates independent of $M$
which allow us to address difficulty \ref{pb:apriori}.
These a priori estimates
imply compactness results
and play an essential role
in the proofs of existence in
Section \ref{sec:ExiUni}.

\subsection{A priori estimates on $u$}
\label{subsc:aprioriu}
When we consider MFG without interactions through controls
and with bounded coupling function and terminal cost,
we can apply the maximum principle on parabolic differential
equations to \eqref{eq:HJB2} below and get an a priori estimates
of $\norminf{u}$ which only depends on the constants in the assumptions.
However, for MFGC systems and more generally for HJB equations
with non-local interactions in $\nabla_xu$,
it is not possible to get such a strong a priori estimate
directly from the maximum principle.
Instead we get
\eqref{eq:MPuint}
and \eqref{eq:MPuinf}
which involve
non-local quantities depending on $\nabla_xu$.
\begin{lemma}
    \label{lem:MPu}
    Under assumptions
    \ref{hypo:Hp},
    \ref{hypo:gregx},
    \ref{hypo:Hh},
    \ref{hypo:Hinvert},
    \ref{hypo:Hcontrac},
    and $q_0\leq q'$,
    for $\th\in(0,1)$
    $u$ satisfies,
    \begin{equation}
        \label{eq:MPuint}
        \norminf{u}
        \leq
        C_0\lp 1+T\rp 
        +\frac{\ll_2C_0^{q'}}
        {(1-\ll_0)^{q'}}
        \lp\th^{1-q'}T
        +\lp1-\th\rp^{1-q'}
        \int_0^T\int_{\TT^d}
        \labs \nabla_xu\rabs^{q}
        dm(t,x)dt\rp,
    \end{equation}
    where $\ll_2$ is defined in \ref{hypo:Hh}.
    More generally, for any $q_0\in[1,\infty]$
    $u$ satisfies,
    \begin{equation}
        \label{eq:MPuinf}
        \norminf{u}
        \leq
        C_0\lp 1+T\rp
        +\frac{\ll_2C_0^{q'}}
        {(1-\ll_0)^{q'}}
        \lp\th^{1-q'}T
        +\lp1-\th\rp^{1-q'}
        \norminf[q]{\nabla_xu}\rp.
    \end{equation}
    
\end{lemma}

\begin{proof}
    Here, we can rewrite \eqref{eq:HJBM} in
    the following way,
    \begin{equation}
        \label{eq:HJB2}
        -\ptt u(t,x)
        - \nu\Delta u(t,x)
        + \lb \int_0^1 H_p(x,s\nabla_x u,\mu(t))ds\rb
        \cdot \nabla_x u(t,x)
        =
        -H(x, 0,\mu(t)),
    \end{equation}
    for $\lp t,x\rp\in(0,T)\times\TT^d$.
    The maximum principle for parabolic second-order
    equation applies to $u$ and $-u$,
    \begin{equation}
        \label{eq:MPu}
        \norminf{u}
        \leq
        \norminf{u(T,\cdot)}
        +\int_0^T\norminf{H\lp\cdot,0,\mu(t)\rp}dt.
    \end{equation}
    Moreover,
    $|H\lp x,0, \mu(t)\rp|
    \leq
    C_0+\ll_2\LL_{q_0}\lp\mu(t)\rp^{q'}$
    and $|u(T,x)|\leq C_0$
    come from
    \ref{hypo:Hh} and \ref{hypo:gregx}, respectively.
    We combine the latter inequalities
    with \eqref{eq:boundLq0q'int}
    and \eqref{eq:boundLq0q'inf}
    to get 
    \eqref{eq:MPuint} when $q_0\leq q'$,
    and \eqref{eq:MPuinf} respectively.
   \end{proof}

The non-local term in \eqref{eq:MPuint} involving $\nabla_xu$
corresponds roughly speaking to an energy.
Moreover this is a quantity that naturally appears in MFG literature
thanks to duality properties in
the forward-backward systems \eqref{eq:MFGtrad},
\eqref{eq:MFGC},
or \eqref{eq:MFGCM}.
More precisely,
the FPK equations is the dual equation
of the linearized HJB equation with respect to $u$.
Lemma \ref{lem:NRJ} provides an a priori estimate of this quantity.

\begin{lemma}
    \label{lem:NRJ}
    Under assumptions
    \ref{hypo:Hp},
    \ref{hypo:gregx},
    \ref{hypo:Hpower},
    \ref{hypo:Hinvert},
    \ref{hypo:Hcontrac},
    and $q_0\leq q'$,
    the following inequality is satisfied,
    \begin{multline}
        \label{eq:NRJ}
        \int_0^T \int_{\TT^d}
        \labs \nabla_x u\rabs^{q}dm(t,x)dt
        \leq
        \\
        \lp1-\frac{\ll_1C_0^{q'}}
        {(1-\th)^{q'-1}(1-\ll_0)^{q'}}\rp^{-1}
        \lp
        C_0\norminf{u}
        +C_0^2(1+T)
        +\frac{\ll_1C_0^{q'}T}{\th^{q'-1}(1-\ll_0)^{q'}}
        \rp,
    \end{multline}
    for any $\th\in(0,1)$ such that
    $\ll_1<\frac{(1-\th)^{q'-1}(1-\ll_0)^{q'}}{C_0^{q'}}$.
\end{lemma}

\begin{proof}
    We multiply \eqref{eq:HJBM} by $-m$
    and \eqref{eq:FPKM} by $u$;
    we add up and integrate 
    over $(0,T)\times\TT^d$
    the resulting quantities;
    after performing some integrations by part,
    we obtain
    \begin{multline*}
        \int_0^T\int_{\TT^d}
        \lb H_p\lp x,\nabla_xu(t,x),\mu(t)\rp\cdot \nabla_xu
        - H\lp x,\nabla_xu(t,x),\mu(t) 
        \rp\rb dm(t,x)dt \\
        =
        \int_{\TT^d}
        u(0,x)dm^0(x) 
        - \int_{\TT^d}
        g(x,m(T))dm(T,x),
    \end{multline*}
    that we can combine with
    \ref{hypo:Hpower} and \ref{hypo:gregx}
    to get,
    \begin{equation*}
        C_0^{-1}\int_0^T \int_{\TT^d}
        \labs \nabla_x u\rabs^{q}dm(t,x)dt
        \leq
        \norminf{u}
        +C_0(1+T)
        +C_0^{-1}\ll_1\int_0^T\LL_{q_0}\lp\mu(t)\rp^{q'} dt.
    \end{equation*}
    We integrate \eqref{eq:boundLq0q'int}
    over $(0,T)$,
    \begin{equation*}
        \int_0^T\LL_{q_0}\lp\mu(t)\rp^{q'} dt
        \leq
        \frac{C_0^{q'}}{(1-\ll_0)^{q'}}
        \lp \th^{1-q'}T
        +(1-\th)^{1-q'}
        \int_0^T\int_{\TT^d}
        \labs \nabla_x u\rabs^{q}dm(t,x) dt\rp,
    \end{equation*}
    where we can choose $\th\in(0,1)$ such that
    $\ll_1<\frac{(1-\th)^{q'-1}(1-\ll_0)^{q'}}{C_0^{q'}}$,
    since $\ll_1$ satisfies the inequality in \ref{hypo:Hpower}.
    The latter three inequalities imply \eqref{eq:NRJ}.
\end{proof}

Roughly speaking,
Lemma \ref{lem:MPu} with $q_0\leq q'$
and Lemma \ref{lem:NRJ}
provide opposite inequalities
which may become complementary
under a smallness condition on the parameters,
implying a
uniform estimate on $\norminf{u}$.
This condition is explicitely given in
the following corollary.

\begin{corollary}
    \label{cor:small_param}
    Under Assumptions
    \ref{hypo:Hp},
    \ref{hypo:gregx},
    \ref{hypo:Hh},
    \ref{hypo:Hpower},
    \ref{hypo:Hinvert},
    \ref{hypo:Hcontrac},
    $q_0\leq q'$,
    and $\ll_1+C_0\ll_2<\frac{(1-\ll_0)^{q'}}{C_0^{q'}}$,
    $u$ is bounded by a quantity which only
    depends on the constants in 
    the assumptions.
\end{corollary}

\begin{proof}
    Combing \eqref{eq:MPuint} and \eqref{eq:NRJ}
    results in,
    \begin{align*}
        \norminf{u}
        &\leq
        \frac{\ll_2C_0^{q'}}
        {(1-\th)^{q'-1}(1-\ll_0)^{q'}}
        C_0\lp1-\frac{\ll_1C_0^{q'}}
        {(1-\th)^{q'-1}(1-\ll_0)^{q'}}\rp^{-1}
        \norminf{u}
        +C_{\th}
        \\
        &\leq
        C_0\ll_2
        \lp\frac{(1-\th)^{q'-1}(1-\ll_0)^{q'}}
        {C_0^{q'}}
        -\ll_1\rp^{-1}
        \norminf{u}
        +C_{\th}
    \end{align*}
    where $\th\in(0,1)$ may be chosen such that
    $\ll_1+C_0\ll_2<\frac{(1-\th)^{q'-1}(1-\ll_0)^{q'}}{C_0^{q'}}$,
    and $C_{\th}$ is a positive constant depending on
    the constants in the assumptions and $\th$.
    This implies
    \begin{equation*}
        \norminf{u}
        \leq
        \lp1-C_0\ll_2
        \lp\frac{(1-\th)^{q'-1}(1-\ll_0)^{q'}}
        {C_0^{q'}}
        -\ll_1\rp^{-1}\rp
        C_{\th},
    \end{equation*}
    where $C_0\ll_2\lp\frac{(1-\th)^{q'-1}(1-\ll_0)^{q'}}
    {C_0^{q'}}-\ll_1\rp^{-1}<1$,
    which concludes the proof.
\end{proof}
Let us mention that in the assumption
$\ll_1+C_0\ll_2<\frac{(1-\ll_0)^{q'}}{C_0^{q'}}$
in Corollary \ref{cor:small_param},
the constant $C_0$ in the left-hand side comes from
the $C_0^{-1}$ in \ref{hypo:Hpower},
and the $C_0$ in the right-hand side comes
from \ref{hypo:Hinvert}.

\subsection{A priori estimates on $m$}
\label{subsec:apriorim}
In order for the HJB equations \eqref{eq:HJB}
and \eqref{eq:HJBM} to admit classical
solutions, we want $\mu$ to be regular in time.
Since $m$ is the marginal of $\mu$ with respect
to $\TT^d$, we first prove that $m$
is regular in the following lemma.
Moreover, we also prove
that $m$ stays positive,
which is required 
in Lemma \ref{lem:FPmu}
to obtain time regularity on $\mu$.
\begin{lemma}
    \label{lem:mbounded}
    Under assumptions
    \ref{hypo:Hp},
    \ref{hypo:mreg},
    \ref{hypo:Hinvert},
    \ref{hypo:Hcontrac},
    $m$ is in
    $C^{\frac{\bb}2,\bb}\lp[0,T]\times\TT^d;\RR\rp$
    for $\bb\in(0,\bb_0)$ and its 
    $C^{\frac{\bb}2,\bb}\lp[0,T]\times\TT^d,\RR\rp$-norm
    can be estimated from above
    by a constant which depends
    on $\norm{m_0}{C^{\bb_0}}$,
    $\norminf{\nabla_xu}$, $\bb$
    and the constants in the assumptions.

    Furthermore, $m$ is positive everywhere and 
    admits a positive lower bound which only
    depends on $\norminf{m_0^{-1}}$,
    $\norminf{\nabla_xu}$
    and the constants in the assumptions.
\end{lemma}

\begin{proof}
    The distribution of agents $m$ satisfies
    the second-order parabolic FPK equation
    \eqref{eq:FPKM},
    which is supplemented with a $\bb_0$-Hölder
    continuous initial condition.
    Theorem $2.1$ section $V.2$ in \cite{MR0241822}
    states that $m$ is uniformly bounded by a constant
    which depends on $\norminf{m_0}$
    and $\norminf{H_p\lp\cdot,\nabla_xu,\mu\rp}$.
    This, \eqref{hypo:Hinvert} and \eqref{eq:boundLq0}
    yield that $mH_p\lp\cdot,\nabla_xu,\mu,m\rp$
    is bounded by a constant which depends on
    $\norminf{m_0}$, $\norminf{\nabla_xu}$
    and the constant of the assumptions.
    Finally, Theorem $6.29$ in \cite{MR1465184}
    yields that
    $m\in C^{\frac{\bb}2,\bb}\lp[0,T]\times\TT^d\rp$
    for $\bb\in(0,\bb_0)$,
    and its associated norm can be estimated from
    above by a constant which depends on
    $\norm{m_0}{C^{\bb_0}}$,
    $\norminf{\nabla_xu}$, $\bb$
    and the constants in the assumptions.

    We define
    $T_{\ee}=\inf\lp\lc t\in[0,T], \norminf{m(t)^{-1}}\leq\ee\rc\cup\lc T\rc\rp$,
    for $0<\ee<\norminf{m_0^{-1}}$.
    In particular $T_{\ee}$ is positive,
    since we proved in the latter paragraph
    that $m$ is continuous.
    On $[0,T_{\ee}]\times\TT^d$ we define the function $n$
    by $n=m^{-1}$, it satisfies the following partial differential
    equation in the sense of viscosity,
    \begin{equation*}
        \ptt n
        -\nu\Delta n
        -\divo\lp\aa n\rp
        +2\aa\cdot\nabla_xn
        =
        -2\nu\frac{\labs\nabla_xn\rabs^2}{n},
    \end{equation*}
    supplemented with the initial condition $n(0)=m_0^{-1}$,
    where $\aa(t,x)=-H_p\lp x,\nabla_xu(t,x),\mu(t)\rp$
    for $(t,x)\in[0,T]\times\TT^d$.
    We define $\nt$ as the unique weak solution
    of the following partial differential
    equation defined on $[0,T]\times \TT^d$,
    \begin{equation}
        \label{eq:PDEnt}
        \ptt \nt
        -\nu\Delta \nt
        -\divo\lp\aa \nt\rp
        +2\aa\cdot\nabla_x\nt
        =
        0,
    \end{equation}
    supplemented with the initial condition $\nt(0)=m_0^{-1}$.
    Theorem $2.1$ section $V.2$ in \cite{MR0241822}
    states that $\nt$ is bounded from above by
    a constant which depends on $\norminf{m_0^{-1}}$,
    $\norminf{\aa}$ and $T$.
    Moreover, $n$ is a subsolution of
    the restriction of \eqref{eq:PDEnt}
    to $[0,T_{\ee}]\times\TT^d$,
    with the same initial condition as $\nt$.
    Therefore, by a comparison argument for
    second-order parabolic equations in divergence form
    (Theorem $9.7$ in \cite{MR1465184} for instance),
    $n$ and $\nt$ satisfy $n\leq \nt$.
    This implies that there exists $C$ a positive constant
    independent of $T_{\ee}$,
    such that $\norminf{n}\leq C$.
    We conclude the proof by taking
    $\ee=2^{-1}C^{-1}$
    and recalling that $\norminf{\aa}$ can be estimated
    from above using \ref{hypo:Hinvert} and \eqref{eq:boundLq0}.
\end{proof}

\subsection{A priori estimates on derivatives of $u$}
\label{subsec:apriorigrad}
Bernstein methods are useful tools when
studying HJB equations or MFG systems.
They allow one to obtain a priori estimates on $\nabla_xu$
by considering the partial differential equations
satisfied by some well-chosen functions depending on
$u$ and $\nabla_xu$.
See for example
the video of the lecture
of P.L. Lions on November the $23$rd $2018$
\cite{Lions_video},
in which Bernstein estimates are derived
for MFG systems without interactions through controls.
More precisely, P.L. Lions used
the function defined by $\labs\nabla_xu\rabs^2e^{-\eta u}$,
for small $\eta$.
Here this method might work only
if we knew a uniform estimates on $\norminf{u}$
and if $q=2$.
After significant changes
in the latter method,
we can derive an estimate on $u$
which is weaker than the one for MFG
without interactions through controls.
Namely, we state that $\norminf{\nabla_xu}$
is bounded by a quantity that depends linearly
on $\norminf{u}$ by studying the functions $w$ and $\vp$
defined in \eqref{eq:defphi} below.
To our knowledge,
such estimates for systems of MFG
with nonlocal dependency on $\nabla_xu$
(or more generally for MFG systems
in which we do not have a uniform
a priori estimate on $u$)
are new in the literature.
We believe that this result may hold for more general
HJB equations with nonlocal dependency on
$\nabla_xu$.

\begin{lemma}
    \label{lem:Bernstein}
    Under assumptions
    \ref{hypo:Hp},
    \ref{hypo:gregx},
    \ref{hypo:Hx},
    \ref{hypo:Hpower},
    \ref{hypo:Hinvert}
    and
    \ref{hypo:Hcontrac},
    there exists $C>0$ depending only on the constants
    of the assumptions, such that
    \begin{equation}
        \label{eq:Bernstein}
        \norminf{\nabla_xu(t)}
        \leq
        C\lp 1+\max_{t\leq s\leq T}\norminf{u}\rp,
    \end{equation}
    for any $t\in[0,T]$.
\end{lemma}

\begin{proof}
    In what follows, we only prove that \eqref{eq:Bernstein}
    holds for $t=0$, however the proof does not use
    additional information available at $t=0$
    (the initial condition on $m$ for example),
    so it can be repeated for any $t\in[0,T]$
    and the constant $C$ in \eqref{eq:Bernstein}
    does not depend on $t$.

    Here we wish to differentiate \eqref{eq:HJBM}
    with respect to $x$;
    however we did not assume in Definition
    \ref{def:sol} enough regularity on $u$ for
    such an operation to have sense pointwisely
    on $(0,T)\times\TT^d$.
    Especially the time derivative of $\nabla_xu$
    and the third derivatives of $u$ with respect to $x$
    are not required to exist.
    This leads us to introducing
    $\rho\in C^{\infty}\lp[-\frac12,\frac12)^d \rp$
    a non-negative mollifier
    such that $\rho(x)=0$ if $|x|\geq \frac14$ 
    and $\int_{\RR^d}\rho(x)dx=1$.
    We introduce
    $\rho^{\dd}=\dd^{-d}\rho\lp \frac{\cdot}{\dd}\rp$
    and 
    $u^{\dd}(t)=\rho^{\dd}\star u(t)$,
    for any $0<\dd<1$ and $t\in [0,T]$,
    where $\star$ denotes the convolution operator.
    
    Thus $u^{\dd}$ depends smoothly on the state variable
    and its partial derivatives in space
    at any order have the same regularity in time as $u$,
    moreover it solves
    the following partial differential equation
    with final condition,
\begin{equation}
    \label{eq:ueps}
\left\{
\begin{aligned}
    &-\ptt u^{\dd}(t,x)
    - \nu\Delta u^{\dd}(t,x)
    + \rho^{\dd}\star
    \lp H(\cdot, \nabla_xu(t,\cdot),\mu(t))\rp(x) 
    =
    0
    &\text{ in } (0,T)\times\TT^d,\\
    &u^{\dd}(T,x)
    =
    \rho^{\dd}\star
    \lp g(\cdot,m(T,\cdot))\rp(x) 
    &\text{ in } \TT^d,\\
\end{aligned}
\right.
\end{equation}
    Let us take the gradient with respect to 
    the state variable of the latter equation
    and the scalar product
    of the resulting equality
    with $\nabla_x u^{\dd}$,
\begin{multline}
\label{eq:DiffHJB}
    -\frac12\ptt \labs \nabla_x u^{\dd}\rabs^2 
    -\nu\nabla_x u^{\dd} \cdot  \Delta \lp\nabla_x u^{\dd}\rp 
    + \nabla_x u^{\dd}\cdot D^2_{x,x} u^{\dd} H_p\lp x,\nabla_x u^{\dd}, \mu\rp
    \\
    + \nabla_x u^{\dd}\cdot  H_x^{\dd}\lp x,\nabla_x u, \mu\rp 
    =
    \nabla_x u^{\dd}\cdot R^{\dd}(t,x),
\end{multline}
    where $H^{\dd}$ and $R^{\dd}$ are defined by
    \begin{align*}
        &H^{\dd}(x,p,\mu)
        =
        \rho^{\dd}\star
        \lp H(\cdot, p(\cdot),\mu)\rp(x),
        \\
        &R^{\dd}(t,x)
        =
        D_{x,x}^2u^{\dd}H_p\lp x,\nabla_x u^{\dd},\mu\rp
        -\rho^{\dd}\star \lp D^2_{x,x}uH_p\lp\cdot,\nabla_x u,\mu\rp\rp. 
    \end{align*}
By simple calculus, we notice that
\begin{align*}
    \nabla_x \labs \nabla_x u^{\dd}\rabs^2
    &=
    2 D^2_{x,x} u^{\dd}
    \nabla_xu^{\dd},
    \\
    \Delta  \labs \nabla_x u^{\dd}\rabs^2
    &=
    2\nabla_x u^{\dd}\cdot \Delta \lp \nabla_x u^{\dd}\rp
    + 2\labs D^2_{x,x}u^{\dd}\rabs^2,
\end{align*}
that we can combine with \eqref{eq:DiffHJB}
and obtain
\begin{multline}
\label{eq:DiffHJB2}
    -\frac12\ptt \labs \nabla_x u^{\dd}\rabs^2
    -\frac{\nu}2\Delta \labs\nabla_x u^{\dd}\rabs^2
    + \nu\labs D^2_{x,x}u^{\dd}\rabs^2
    + \frac12\nabla_x\labs \nabla_x u^{\dd}\rabs^2
    \cdot H_p\lp x,\nabla_x u^{\dd}, \mu\rp \\
    =
    - \nabla_x u^{\dd}\cdot H^{\dd}_x\lp x, \nabla_x u, \mu\rp 
    +\nabla_x u^{\dd}\cdot R^{\dd}(t,x).
\end{multline}
We define the functions $\vp$ and $w^{\dd}$ by
\begin{equation}
\label{eq:defphi}
    \begin{aligned}
        &\vp(v)
        =
        \exp\lp\exp\lp
        -\lp a +b\norminf[-1]{u}v\rp\rp\rp,
        \text{ for } |v|\leq \norminf{u},
        \\ 
        &w^{\dd}(t,x)
        =
        \vp(u^{\dd}(T-t,x))\labs \nabla_x u^{\dd}\rabs^2(T-t,x),
    \end{aligned}
\end{equation}
    where $a>1$ and $b>0$ are constants that will be defined below.
    The derivatives of $\vp$
    are given by
    \begin{equation}
        \label{eq:diffphi}
        \begin{aligned}
        &\vp'(v)
        =
        -b\norminf[-1]{u}
        e^{-\lp a +b\norminf[-1]{u}v\rp}
        \vp(v),
        \\
        &\vp''(v)
        =
        b^2\norminf[-2]{u}
        e^{-\lp a +b\norminf[-1]{u}v\rp}
        \lp1+e^{-\lp a +b\norminf[-1]{u}v\rp}\rp
        \vp(v),
        \end{aligned}
    \end{equation}
    which implies that $\vp$ and $\vp'$ satisfy,
    \begin{equation}
        \label{eq:boundphi}
        \begin{aligned}
            &1
            \leq
            \vp(v)
            \leq e^{e^{-a+b}},
            \\
            &b\norminf[-1]{u}
            e^{-a-b}
            \leq
            \frac{\labs\vp'(v)\rabs}{\vp(v)} 
            \leq
            b\norminf[-1]{u}
            e^{-a+b}.
        \end{aligned}
    \end{equation}
    Roughly speaking, we introduced $a$ and $b$
    in order to have
    $\norminf{\vp}\norminf{\vp^{-1}}$
    and $\norminf{\vp'}\norminf{(\vp')^{-1}}$
    as close as possible to $1$.
    This will be achieved by 
    taking $a$ large enough,
    and $b$ small enough.

    For simplicity of the notations,
    we will omit to write the argument
    of $\vp$ since it is always $u^{\dd}$.

    The derivatives of $w^{\dd}$
    verify the following equalities,
    \begin{align*}
        -&\vp\ptt\labs\nabla_xu^{\dd}\rabs^2
        =
        \ptt w^{\dd} 
        +\frac{\vp'}{\vp}w^{\dd}\ptt u^{\dd},
        \\
        &\vp\nabla_x\labs\nabla_xu^{\dd}\rabs^2
        =
        \nabla_x w^{\dd} 
        -\frac{\vp'}{\vp}w^{\dd}\nabla_xu^{\dd},
        \\
        &\vp\Delta\labs\nabla_xu^{\dd}\rabs^2
        =
        \Delta w^{\dd} 
        -\frac{\vp'}{\vp}w^{\dd}\Delta u^{\dd}
        +2\frac{\vp'}{\vp}\nabla_xw^{\dd}\cdot\nabla_xu^{\dd}
        -\frac{\vp''\vp-2\lp\vp'\rp^2}{\vp^3}\lp w^{\dd}\rp^2.
    \end{align*}
    We multiply \eqref{eq:DiffHJB2} by $2\vp$ and
    use the latter equalities in the resulting relation,
    \begin{multline}
        \label{eq:firstPDEw}
        \ptt w^{\dd} -\nu\Delta w^{\dd}
        + \nabla_xw^{\dd}\cdot  H_p\lp x,\nabla_x u^{\dd}, \mu\rp 
        -2\nu\frac{\vp'}{\vp} \nabla_xw^{\dd}\cdot \nabla_xu^{\dd} 
        +2\nu\vp\labs D^2_{x,x}u^{\dd}\rabs^2
        \\
        =
        \frac{\vp'}{\vp}w^{\dd}\lb
        -\ptt u^{\dd}-\nu\Delta u^{\dd}
        +\nabla_xu^{\dd}\cdot H_p\lp x,\nabla_xu^{\dd},\mu\rp\rb
        -\nu\frac{\vp''\vp-2\lp\vp'\rp^2}{\vp^3}\lp w^{\dd}\rp^2
        \\
        -2\vp\nabla_xu^{\dd}\cdot H_x^{\dd}\lp x,\nabla_xu,\mu\rp
        +2\vp\nabla_x u^{\dd}\cdot R^{\dd}(t,x).
    \end{multline}
    We can rewrite the first line of \eqref{eq:ueps} in
    the following way,
    \begin{equation*}
        -\ptt u^{\dd}
        - \nu\Delta u^{\dd}
        =
        -H\lp x,\nabla_xu^{\dd},\mu\rp
        -Q^{\dd},
    \end{equation*}
    where $Q^{\dd}$ is defined by,
    \begin{equation*}
        Q^{\dd}(t,x)
        =
        H^{\dd}\lp x , \nabla_x u(t), \mu(t)\rp
        -H\lp x,\nabla_xu^{\dd}(t,x) ,\mu(t)\rp.
    \end{equation*}
    This and \eqref{eq:firstPDEw} imply that
    \begin{multline}
        \label{eq:ubound_PDEw}
        \ptt w^{\dd} -\nu\Delta w^{\dd}
        + \nabla_xw^{\dd}\cdot  H_p\lp x,\nabla_x u^{\dd}, \mu\rp 
        -2\nu\frac{\vp'}{\vp} \nabla_xw^{\dd}\cdot \nabla_xu^{\dd} 
        +2\nu\vp\labs D^2_{x,x}u^{\dd}\rabs^2
        \\
        =
        \frac{\vp'}{\vp}w^{\dd}\lb
        \nabla_xu^{\dd}\cdot H_p\lp x,\nabla_xu^{\dd},\mu\rp
        -H\lp x,\nabla_xu^{\dd},\mu\rp\rb
        -\nu\frac{\vp''\vp-2\lp\vp'\rp^2}{\vp^3}\lp w^{\dd}\rp^2
        \\
        -2\vp\nabla_xu^{\dd}\cdot H_x^{\dd}\lp x,\nabla_xu,\mu\rp
        +2\vp\nabla_x u^{\dd}\cdot R^{\dd}(t,x)
        -\vp'\labs\nabla_xu^{\dd}\rabs^2Q^{\dd}(t,x).
    \end{multline}
    In the following we will estimate from above
    the right-hand side of the latter expression.
    We notice that the second term
    of the right-hand side is negative since
    \begin{equation}
        \label{eq:condphi}
        \vp''\vp-2\lp\vp'\rp^2
        \geq
        0.
    \end{equation}
    We notice that
    $R^{\dd}$ and $Q^{\dd}$
    are uniformly convergent to $0$
    as $\dd$ tends to $0$,
    so we can assume that,
    \begin{equation}
        \label{eq:RQlow}
        \norminf{R^{\dd}}
        +\norminf{Q^{\dd}}
        \leq 
        \frac{\e}
        {2e\norminf{\nabla_xu}
        +\norminf{\vp'}
        \norminf[2]{\nabla_xu}},
    \end{equation}
    for $\dd$ small enough and
    depending on $\ee>0$.

    The first term in the last line of \eqref{eq:ubound_PDEw}
    can be bounded using \ref{hypo:Hx},
    \begin{equation}
        \label{eq:boundHx}
        -2\vp\nabla_xu^{\dd}\cdot H_x^{\dd}\lp x,\nabla_xu,\mu\rp
        \leq
        2C_0\vp\norminf{\nabla_xu^{\dd}}
        \lp 1+
        \norminf[q]{\nabla_xu^{\dd}}
        +\LL_{q_0}\lp\mu(t)\rp^{q'}\rp.
    \end{equation}
    In fact we are going to use the latter inequality 
    to obtain \eqref{eq:ubound_PDIw} below, only by noticing that
    using \eqref{eq:boundLq0q'inf},
    the right-hand side involves only terms with exponents
    in $\norminf{w^{\dd}}$
    or $\norminf{w^{0}}$
    not larger than $\frac{1+q}2$.

    Then we use \ref{hypo:Hpower}
    on the first term
    of the right-hand side of \eqref{eq:ubound_PDEw}
    since $\vp'<0$,
    \begin{equation*}
        \frac{\vp'}{\vp}w^{\dd}\lb
        \nabla_xu^{\dd}\cdot H_p\lp x,\nabla_xu^{\dd},\mu\rp
        -H\lp x,\nabla_xu^{\dd},\mu\rp\rb
        \leq
        -C_0^{-1}
        \frac{\vp'}{\vp^{1+\frac{q}2}}
        \lp w^{\dd}\rp^{1+\frac{q}2}
        +C_0\frac{\labs\vp'\rabs}{\vp}w^{\dd}
        +C_0^{-1}\ll_1\frac{\labs\vp'\rabs}{\vp}w^{\dd}
        \LL_{q_0}\lp\mu(t)\rp^{q'}.
    \end{equation*}
    The term involving $-\lp w^{\dd}\rp^{1+\frac{q}2}$
    is a key element in this proof.
    On the one hand, it will allow us to cancel the term in 
    $w^{\dd}\LL_{q_0}\lp\mu(t)\rp^{q'}$.
    On the other hand, we will
    use the fact that it has a larger exponent
    than any of the remaining terms.

    From \eqref{eq:boundLq0q'inf} and \eqref{eq:boundphi},
    we obtain
    \begin{equation*}
        \frac{\labs\vp'\rabs}{\vp}
        \LL_{q_0}\lp\mu(t)\rp^{q'}
        \leq
        b\norminf[-1]{u}e^{-a+b}
        \frac{C_0^{q'}}{(1-\ll_0)^{q'}}
        \lp \th^{1-q'}
        +(1-\th)^{1-q'}
        \norminf[\frac{q}2]{w^0}\rp,
    \end{equation*}
    where $\th\in(0,1)$ will be defined below.
    Then \eqref{eq:boundphi} implies,
    \begin{equation*}
        \frac{\labs\vp'\rabs}{\vp^{1+\frac{q}2}}
        \geq
        b\norminf[-1]{u}e^{-a-b}e^{-\frac{q}2e^{-a+b}}
    \end{equation*}
    Combining the latter six inequalities,
    \eqref{eq:ubound_PDEw},
    and the fact that $\norminf{w^{\dd}}\leq\norminf{w^{0}}$,
    we obtain the following partial differential
    inequality,
    \begin{multline}
        \label{eq:ubound_PDIw}
        \ptt w^{\dd} -\nu\Delta w^{\dd}
        + \nabla_xw^{\dd}\cdot  H_p\lp x,\nabla_x u^{\dd}, \mu\rp 
        -2\nu\frac{\vp'}{\vp} \nabla_xw^{\dd}\cdot \nabla_xu^{\dd} 
        \\
        \leq
        -C_0^{-1}
        b\norminf[-1]{u}e^{-a-b}e^{-\frac{q}2e^{-a+b}}
        \lp w^{\dd}\rp^{1+\frac{q}2}
        +b\norminf[-1]{u}e^{-a+b}
        \frac{\ll_1C_0^{q'-1}}{(1-\th)^{1-q'}(1-\ll_0)^{q'}}
        \norminf[1+\frac{q}2]{w^0}
        \\
        +\ee
        +C_{a,b,\th}\lp1+\norminf[-1]{u}\rp
        \lp 1+
        \norminf[\frac{1+q}2]{w^0}\rp
    \end{multline}
    where $C_{a,b,\th}$ is a positive constant which only depends
    on the constants in the assumptions and in $(a,b,\th)$.
    We systematically used the inequality
    $\norminf[r]{w^0}\leq 1+\norminf[\frac{1+q}2]{w^0}$
    on every term of the form
    $\norminf[r]{w^0}$ with $0<r<\frac{1+q}2$.

    Let us mention the following result:
    the function $y^+$
    defined by $y^+=\max\lp y_0,K^{-\frac1k}\norminf[\frac1k]{f}\rp$
    is a super-solution of the following differential equation,
    \begin{equation*}
        \lc
        \begin{aligned}
            y'(t)
            &=
            -Ky(t)^k
            +f(t)
            \\
            y(0)
            &=
            y_0
        \end{aligned}
    \right.
    \end{equation*}
    posed on $[0,T]$,
    where $k$ and $y_0$ are positive constants and
    $f$ is a bounded positive function.

    This and $\norminf{w(0)}\leq eC_0^2$ 
    which comes from \ref{hypo:gregx}
    and \eqref{eq:boundphi},
    yield that a super-solution to \eqref{eq:ubound_PDIw}
    is given by
    \begin{equation*}
        \lb
        \frac{\ll_1C_0^{q'}e^{2b}e^{\frac{q}2e^{-a+b}}}
        {(1-\th)^{q'-1}(1-\ll_0)^{q'}}\norminf[1+\frac{q}2]{w^0}
        +C_{a,b,\th}\lp 1+ \norminf{u}\rp
        \lp 1+\norminf[\frac{1+q}2]{w^0}+\ee\rp
        \rb^{\frac{2}{2+q}},
    \end{equation*}
    where we replace $C_{a,b,\th}$ with 
    $C_{a,b,\th}+\lp eC_0^2\rp^{1+\frac{q}2}$.

    From a comparison argument for parabolic second-order equation,
    $w^{\dd}$ is not larger than the latter expression.
    This result holds for $w^0$ by letting $\dd$ and $\ee$ tend to $0$,
    thus $w^0$ verifies the following inequality,
    \begin{equation*}
        \norminf[1+\frac{q}2]{w^0}
        \leq
        \frac{\ll_1C_0^{q'}e^{2b}e^{\frac{q}2e^{-a+b}}}
        {(1-\th)^{q'-1}(1-\ll_0)^{q'}}
        \norminf[1+\frac{q}2]{w^0}
        +C_{a,b,\th}\lp 1+ \norminf{u}\rp
        \lp 1+\norminf[\frac{1+q}2]{w^0}\rp.
    \end{equation*}
    By \ref{hypo:Hpower},
    we can choose $a>1$ large enough,
    $b>0$ and $\th\in(0,1)$ small enough
    such that 
    $\frac{\ll_1C_0^{q'}e^{2b}e^{\frac{q}2e^{-a+b}}}
    {(1-\th)^{q'-1}(1-\ll_0)^{q'}}<1$.
    This implies
    \begin{equation}
        \label{eq:boundw}
        \norminf[1+\frac{q}2]{w^0}
        \leq
        C_{a,b,\th}\lp 1+ \norminf{u}\rp
        \lp 1+\norminf[\frac{1+q}2]{w^0}\rp,
    \end{equation}
    where we increased $C_{a,b,\th}$ into 
    $\lp1-\frac{\ll_1C_0^{q'+1}e^{2b}e^{\frac{q}2e^{-a+b}}}
    {(1-\th)^{q'-1}(1-\ll_0)^{q'}}\rp^{-1}
    C_{a,b,\th}$.

    We make out two cases:
    the first case is when
    $\norminf[\frac12]{w^0}\leq 2C_{a,b,\th}\lp 1+\norminf{u}\rp$.
    The second case is when 
    $\norminf[\frac12]{w^0}> 2C_{a,b,\th}\lp 1+\norminf{u}\rp$.
    In the latter case,
    \eqref{eq:boundw} implies that
    $\norminf[1+\frac{q}2]{w^0}
    \leq\frac12\norminf[\frac12]{w^0}
    \lp1+
    \norminf[\frac{1+q}2]{w^0}\rp$,
    which implies that $\norminf{w^0}\leq 1$.
    Therefore, in any of the two latter cases
    we obtain
    \begin{equation*}
        \norminf[\frac12]{w^0}
        \leq
        1+2C_{a,b,\th}\lp 1+\norminf{u}\rp.
    \end{equation*}
    This and \eqref{eq:boundphi}
    yield \eqref{eq:Bernstein} when $t=0$,
    this concludes the proof.
\end{proof}

Now, we can combine the estimates obtained in this section
with classical results on parabolic second-order equations
and get further estimates of $u$ and its derivatives and
on $m$.

\begin{lemma}
    \label{lem:ubound}
    Assume
    \ref{hypo:Hp},
    \ref{hypo:gregx},
    \ref{hypo:Hx},
    \ref{hypo:Hpower},
    \ref{hypo:Hinvert},
    \ref{hypo:Hcontrac}
    and
    \ref{hypo:Hlip}. 
    The function $u$ is in
    $C^{1+\frac{\bb}{2},2+\bb}
    \lp [0,T]\times\TT^d\rp$
    for any $\bb\in\lp 0,\bb_0^2\rp$, 
    where $\bb_0$ was introduced in the assumptions.
    Its $C^{1+\frac{\bb}{2},2+\bb}$-norm
    can be bounded by a quantity depending only on 
    $\norminf{u}$, $\bb$,
    and the constants in the assumptions.
\end{lemma}

\begin{proof}
\noindent
Lemma \ref{lem:Bernstein} states that $\norminf{\nabla_xu}$
is bounded by a quantity which depends on $\norminf{u}$
and the constants in the assumptions.
So is $\LL_{q_0}(\mu)$ by \eqref{eq:boundL}.
Then $u$ is the solution of the heat equation 
with a right-hand side equal to $-H\lp x,\nabla_xu,\mu\rp$ 
which is bounded in $L^{\infty}$.
Classical results (see for example Theorem $6.48$ in \cite{MR1465184})
state that
for any $\bb\in(0,1)$,
the $C^{\frac12+\frac{\bb}2,1+\bb}$-norm of $u$ is bounded 
by a constant which depends 
on the $L^{\infty}$-norm of the right-hand side,
the terminal condition, and $\bb$.
%

Lemma \ref{lem:mbounded} yields that
$m$ is in $C^{\frac{\bb}2,\bb}\lp[0,T]\times\TT^d\rp$
for $\bb\in(0,\bb_0)$, is positive,
and that both its $C^{\frac{\bb}2,\bb}\lp[0,T]\times\TT^d\rp$-norm
and its lower bound depend
on $\norminf{u}$, $\norminf{m_0^{-1}}$,
$\bb$, and the constant of the assumptions.

Therefore, Lemma \ref{lem:FPmu} yields
that $\lb(t,x)\mapsto \aa^{\mu(t)}(x)\rb
\in C^{\frac{\bb\bb_0}2,\bb\bb_0}\lp[0,T]\times\TT^d;\RR^d\rp$.
    From \ref{hypo:Hp},
    $H$ is locally Lipschitz continuous with respect to $(x,p)$.
    This and \ref{hypo:Hlip} imply that
    $\lb(t,x)\mapsto H\lp t,\nabla_xu\lp t,x\rp,\mu(t)\rp\rb
    \in C^{\frac{\bb\bb_0}2,\bb\bb_0}\lp[0,T]\times\TT^d\rp$.
    Thus $u$ is the solution of the backward heat equation
    with a right-hand side in
    $C^{\frac{\bb\bb_0}{2},\bb\bb_0}$
    supplemented with terminal condition in $C^{2+\bb_0}$.
    Classical results
    (see for instance Theorem $4.9$ in \cite{MR1465184})
    yield that $u$ is in
    $C^{1+\frac{\bb\bb_0}{2},2+\bb\bb_0}$, 
    and its $C^{1+\frac{\bb\bb_0}{2},2+\bb\bb_0}$-norm
    depends on $\normc{g(\cdot,m(T))}{2+\bb_0}$ 
    and the $C^{\frac{\bb\bb_0}{2},\bb\bb_0}$-norm
    of the right-hand side.
    We recall that $\bb$ is any constant in $(0,\bb_0)$.
    The proof of the lemma is complete.

    Following precisely the dependencies
    in the above estimates,
    we obtain that the 
    $C^{1+\frac{\bb\bb_0}{2},2+\bb\bb_0}$ norm of $u$ 
    can be estimated from above by a constant which depends
    on $\norminf{u}$, $\norminf{m_0^{-1}}$,
    $\bb$, and the constants in the assumptions.
\end{proof}
The conclusions of Lemmas
\ref{lem:MPu},
\ref{lem:mbounded},
\ref{lem:Bernstein}
and
\ref{lem:ubound}
are summarized
in Lemma \ref{lem:MFGCapriori}.

\section{Existence and uniqueness results
under additional assumptions}
\label{sec:ExiUni}

\subsection{Solving the MFGC systems for $M<\infty$}
\begin{lemma}
    \label{lem:existM}
    Under assumptions
    \ref{hypo:Hp}-\ref{hypo:mreg},
    \ref{hypo:Hx},
    \ref{hypo:Hpower},
    \ref{hypo:Hinvert},
    \ref{hypo:Hcontrac},
    \ref{hypo:Hlip}
    and $M\in(0,\infty)$,
    there exists
    at least one solution
    to \eqref{eq:MFGCM}.
\end{lemma}

\begin{proof}
    For $(u,m)\in
    C^{0,1}\lp [0,T]\times\TT^d;\RR\rp
    \times C^0\lp[0,T];\cP\lp\TT^d\rp \rp$,
    we define
    $\mu^M\in C^0\lp [0,T];\cP \lp\TT^d\times\RR^d\rp\rp$ by
    \begin{equation*}
        \mu^M(t) 
        =
        \lb I_d,
        T_M\lp -H_p\lp \cdot,\nabla_xu(t,\cdot),\mu^M(t)\rp\rp
        \rb{\#}m(t)
        \text{ in } [0,T],
    \end{equation*}
    using Lemma \ref{lem:FPmu}.
    Then we define $u^M$
    as the viscosity solution
    of the following backward
    HJB equation
    with a final condition,
    \begin{equation}
        \label{eq:existM_defuM}
        \lc
        \begin{aligned}
            &-\ptt u^M(t,x)
            - \nu\Delta u^M(t,x)
            +H(x, \nabla_xu^M(t,x),\mu^M(t))
            =
            0
            \\
            &u^M(T,x)
            =
            g(x,m(t)).
        \end{aligned}
        \right.
    \end{equation}
    We can rewrite the first line of the latter
    system in the following way,
    \begin{equation*}
        -\ptt u^M
        - \nu\Delta u^M
        + \nabla_xu^M
        \cdot \int_0^1 
        H_p\lp x, s\nabla_x u^M,\mu^M(t)\rp ds
        =  
        - H\lp x, 0,\mu^M(t)\rp,
    \end{equation*}
    where the right-hand side
    is bounded
    using $\LL_{\infty}\lp\mu(t)\rp\leq M$,
    \ref{hypo:Hh} and \eqref{eq:Jensen}.
    The maximum principle 
    for second-order parabolic equation
    provides that $u^M$ is bounded.
    Here, the proof of Lemma \ref{lem:Bernstein}
    can be repeated to prove that $\norminf{\nabla_xu}$
    is bounded by a constant which depends on $M$ and
    the constants in the assumptions.
    Then with the same argument as in Lemma \ref{lem:ubound},
    $u^M$ is bounded in $C^{\frac12+\frac{\bb}2,1+\bb}$-norm,
    for all $\bb\in(0,1)$.

    We define $m^M$ as the solution in the sense of distributions
    of the following Fokker-Planck-Kolmogorov equation
    with an initial condition,
    \begin{equation*}
        \lc
        \begin{aligned}
            & \ptt m^M_t(t,x)
            - \nu\Delta m^M(t,x)
            +\divo\lp b(t,x)m^M\rp
            =
            0
            &\text{ in } (0,T)\times\TT^d,\\
            &m^M(0)=m_0,
        \end{aligned}
        \right.
    \end{equation*}
    with $b(t,x)=-H_p\lp x,\nabla_xu^M(t,x),\mu^M(t)\rp$
    which is a continuous function with respect to $(t,x)$.
    Using the same arguments as in Lemma \ref{lem:mbounded},
    we get that $m\in C^{\frac{\bb}2,\bb}\lp\TT^d;\RR\rp$
    for $\bb\in(0,\bb_0)$.

    Moreover,
    $\norm{u}{C^{\frac12+\frac{\bb}2,1+\bb}}$,
    $\norm{m}{C^{\frac{\bb}2,\bb}}$
    are bounded by a constant which depends on
    $M$, $\bb$ and the constants in the assumptions.
    The map $(u,m)\mapsto\mu^M$ is continuous 
    from $C^{0,1}\lp [0,T]\times\TT^d;\RR\rp
    \times C^0\lp [0,T];\cP\lp \TT^d\rp\rp$
    to $C^0\lp [0,T];\cP\lp \TT^d\times\RR^d\rp\rp$
    by Lemma \ref{lem:FPV}.
    The map $\lp m,\mu^M\rp\mapsto u^M$ is continuous from
    $C^0\lp [0,T];\cP\lp \TT^d\rp\rp
    \times C^0\lp [0,T];\cP\lp \TT^d\times\RR^d\rp\rp$
    to $C^{0,1}\lp [0,T]\times\TT^d;\RR\rp$
    by the stability of the solutions of viscosity.
    The map $\lp u^M,\mu^M\rp\mapsto m^M$
    is continuous from
    $C^{0,1}\lp [0,T]\times\TT^d;\RR\rp
    \times C^0\lp [0,T];\cP\lp \TT^d\times\RR^d\rp\rp$
    to $C^0\lp [0,T];\cP\lp \TT^d\rp\rp$
    by linearity of the FPK equation.

    Thus the map $(u,m)\mapsto(u^M,m^M)$ 
    is continuous from
    $C^{0,1}\lp [0,T]\times\TT^d;\RR\rp
    \times C^0\lp [0,T];\cP\lp \TT^d\rp\rp$
    to itself.
    Its fixed points are exactly
    the solutions to \eqref{eq:MFGCM}.
    The image of this map
    is a subset of a convex compact set.
    Therefore, there exists a fixed point
    by Schauder theorem,
    see
    \cite{MR1814364}
    Corollary $11.2$.

    Using the same arguments as in
    the proof of Lemma \ref{lem:ubound},
    such a fixed point $u$ satisfies
    $u\in C^{1+\frac{\bb}2,2+\bb}
    \lp [0,T]\times\TT^d;\RR\rp$
    for any $\bb\in\lp0,\bb_0^2\rp$.
\end{proof}

Considering $M<\infty$ in \eqref{eq:MFGCM}
consists of enforcing the condition
$\LL_{\infty}\lp\mu(t)\rp\leq M$,
i.e. the fact that the support of $\mu(t)$ is embedded
in the compact set $\TT^d\times B_{\RR^d}\lp0,M\rp$,
for $t\in[0,T]$.
Therefore, the interactions through controls are uniformly
bounded.
Lemma \ref{lem:existM} relies on that
to state the existence of solutions to 
\eqref{eq:MFGCM}.
For $M=\infty$, we can not obtain such a uniform
estimate by combining only the results of Section
\ref{sec:apriori}.
However if such an estimate exists,
the result of Lemma \ref{lem:existM} holds
for $M=\infty$ and yields the existence
of solutions to \eqref{eq:MFGC}.
More precisely, if
a solution to \eqref{eq:MFGCM}
satisfies
$\LL_{\infty}\lp\mu(t)\rp< M$
for any $t\in[0,T]$,
then it is also a solution to
\eqref{eq:MFGC}.
This is summarized in the following Corollary.

\begin{corollary}
    \label{cor:exist}
    Under the same assumptions as in Lemma \ref{lem:existM},
    if, for any $M>0$, any solution
    $(u,m,\mu)$ to \eqref{eq:MFGCM} 
    satisfies
    $\norminf{u}\leq C$, or
    $\norminf{\nabla_xu}\leq C$,
    for some $C>0$,
    then
    there exists at least one solution to
    \eqref{eq:MFGC}.
\end{corollary}

\begin{proof}
    By Lemma \ref{lem:existM},
    we define $(u,m,\mu)$ as a solution to \eqref{eq:MFGCM}
    for $M\in(0,\infty)$ that will be defined later.
    By Lemma \ref{lem:MFGCapriori},
    assuming 
    that $\norminf{u}$
    is bounded
    is equivalent to assuming that
    $\norminf{\nabla_xu}$
    is bounded.
    Therefore, without loss of generality,
    we can assume that 
    $\norminf{\nabla_xu}\leq C$.
    %
    From \ref{hypo:Hinvert} and \eqref{eq:boundLq0},
    we obtain
    \begin{equation*}
        \norminf{H_p\lp x,\nabla_xu,\mu\rp}
        \leq
        C_0\lp 1+C^{q-1}\rp+
        \frac{\ll_0C_0}{1-\ll_0}\lp 1+C^{q-1}\rp.
    \end{equation*}
    We define $M=1+C_0\lp 1+C^{q-1}\rp+
    \frac{\ll_0C_0}{1-\ll_0}\lp 1+C^{q-1}\rp$,
    then the truncation $T_M$ leaves
    $-H_p\lp \cdot,\nabla_xu^M,\mu\rp$ unchanged.
    Hence $(u,m,\mu)$
    is a solution to \eqref{eq:MFGC}.
\end{proof}

\subsection{Existence results 
    when $q_0\leq q'$}
When $q_0\leq q'$,
we can use integral energy estimates.
More precisely,
inequalities
\eqref{eq:MPuint}
and \eqref{eq:NRJ} hold.
Therefore,
the assumptions under which
we can prove existence should be weaker
than in the case $q_0>q'$
in which we have less estimates
at our disposal.

In particular,
Corollary \ref{cor:small_param}
provides a uniform
estimate on $\norminf{u}$
under suitable assumptions.
Corollary \ref{cor:exist}
then yields the existence
of a solution to \eqref{eq:MFGC}:
hence we may state the following theorem:
\begin{proposition}[\emph{Existence of solution
    with small non-linearities}]
\label{thm:ExiLowParam}
    Under assumptions
    \ref{hypo:Hp}-\ref{hypo:mreg},
    \ref{hypo:Hh}-\ref{hypo:Hpower},
    \ref{hypo:Hinvert},
    \ref{hypo:Hcontrac},
    \ref{hypo:Hlip},
    $q_0\leq q'$,
    and $\ll_1+C_0\ll_2<\frac{(1-\ll_0)^{q'}}{C_0^{q'}}$,
    there exists at least
    one solution to \eqref{eq:MFGC}.
\end{proposition}

Instead of assuming that the multiplicative parameters
are small like in 
Proposition \ref{thm:ExiLowParam};
we suppose in Propositions
\ref{thm:ExiBound}
below the exponent
for the interactions
through controls 
is in fact smaller than the one appearing
in \ref{hypo:Hh}.
\begin{proposition}
    \label{thm:ExiExpo}
    Assume
    \ref{hypo:Hp}-\ref{hypo:mreg},
    \ref{hypo:Hx},
    \ref{hypo:Hpower},
    \ref{hypo:Hinvert},
    \ref{hypo:Hcontrac},
    \ref{hypo:Hlip},
    $q_0\leq q'$,
    and that $H$ satisfies
    \begin{equation*}
        \labs H(x,0,\mu)\rabs
        \leq
        C_0\lp 1 +  \LL_{q_0}\lp\mu\rp^{\qt}\rp,
    \end{equation*}
    for $(x,\mu)\in\TT^d\times
    \cP\lp\TT^d\times\RR^d\rp$,
    where $\qt\in[0,q')$ is a constant.
    There exists a solution
    to \eqref{eq:MFGC}.
\end{proposition}

\begin{proof}
    Let $(u,m,\mu)$ be a solution to \eqref{eq:MFGCM}
    for $M\in(0,\infty)$.
    From \ref{hypo:gregx}, \eqref{eq:MPu} and the new assumption,
    we obtain that,
    \begin{align*}
        \norminf{u}
        &\leq
        C_0\lp 1+T\rp
        +C_0\int_0^T\LL_{q_0}\lp\mu(t)\rp^{\qt} dt
        \\
        &\leq
        C_0\lp 1+T\rp
        +C_0T^{\frac{q'-\qt}{q'}}
        \lp\int_0^T\LL_{q_0}\lp\mu(t)\rp^{q'} dt\rp^{\frac{\qt}{q'}},
    \end{align*}
    where the second line is obtained by a Hölder inequality,
    since $\qt<q'$.
    Let us recall that the inequality
    $\lp a+b\rp^{\frac{\qt}{q'}}
    \leq
    a^{\frac{\qt}{q'}}
    +b^{\frac{\qt}{q'}}$
    holds for any $a,b>0$.
    The latter two inequalities
    and \eqref{eq:boundLq0q'int} with $\th=\frac12$ imply,
    \begin{equation*}
        \norminf{u}
        \leq
        C+
        C\lp\int_0^T\int_{\TT^d}
        \labs\nabla_xu(t,x)\rabs^{q}dm(t,x)dt\rp^{\frac{\qt}{q'}},
    \end{equation*}
    where $C>0$ is a constant which depends on
    the constants in the assumptions.
    This and \eqref{eq:NRJ} yield that,
    \begin{equation*}
        \norminf{u}
        \leq
        C+
        C\norminf[\frac{\qt}{q'}]{u},
    \end{equation*}
    up changing the value of $C$.
    Let us make out two cases:
    the first case is when
    $\norminf{u}\leq \lp2C\rp^{\frac{q'}{q'-\qt}}$.
    The second case is when
    $\norminf{u}> \lp2C\rp^{\frac{q'}{q'-\qt}}$,
    which implies $\norminf{u}\leq C+\frac12\norminf{u}$.
    In any of the two cases,
    $u$ is uniformly bounded with respect to $M$.
    The desired result then stems from
    Corollary \ref{cor:exist}.
\end{proof}

\subsection{Existence results
    which do not need the assumption $q_0<q'$}
Here, we do not make
the assumption $q_0\leq q'$.
We can still obtain an existence result
in the same spirit as the one
provided in Proposition
\ref{thm:ExiExpo}.
In the following proposition,
the exponent for the interactions
through controls is assumed to be smaller than 
the one appearing
in \ref{hypo:Hh}
or in Proposition
\ref{thm:ExiExpo}.
\begin{proposition}
    \label{thm:ExiBound}
    Assume
    \ref{hypo:Hp}-\ref{hypo:mreg},
    \ref{hypo:Hx},
    \ref{hypo:Hpower},
    \ref{hypo:Hinvert},
    \ref{hypo:Hcontrac},
    \ref{hypo:Hlip},
    and that $H$ satisfies
    \begin{equation}
        \label{eq:Exibound}
        \labs H(x,0,\mu)\rabs
        \leq
        C_0\lp 1+\LL_{q_0}\lp\mu\rp^{q'-1}\rp
    \end{equation}
    for any $(x,\mu)\in\TT^d\times
    \cP\lp\TT^d\times\RR^d\rp$.
    There exists a solution
    to \eqref{eq:MFGC}.
\end{proposition}

\begin{proof}
    Take $(u,m,\mu)$
    a solution to \eqref{eq:MFGCM}
    for $M\in(0,\infty)$.    
    Let us combine \eqref{eq:HJB2},
    \eqref{eq:boundLq0q'inf} for $\th=\frac12$,
    \eqref{eq:Bernstein},
    \eqref{eq:Exibound},
    and the inequality
    $\lp a+b\rp^{\frac1q}
    \leq
    a^{\frac1q}+b^{\frac1q}$
    which holds for $a,b>0$;
    this yields
    \begin{equation}
        \label{eq:ExiBound_PDI}        
        -\ptt u
        -\nu\Delta u
        +\nabla_xu\cdot
        \int_0^1
        H_p\lp x,s\nabla_xu(t,x),\mu(t)\rp ds
        \leq
        C\lp1+\max_{t\leq s\leq T}\norminf{u(s)}\rp,
    \end{equation}
    for a constant $C>0$ which depends only
    on the constants in the assumptions.
    We recall that $\norminf{u(T)}\leq C_0$
    by \ref{hypo:gregx}.
    We consider $y_{+},y_-\in C^1\lp [0,T];\RR\rp$
    defined as $y_+(t)=Ct+C_0e^{Ct}$
    and $y_-(t)=-Ct-C_0e^{Ct}$
    such that they are solution
    to the following differential equations
    \begin{align*}
        &\lc
        \begin{aligned}
            y_+'(t)
            &=C(1+y_+(t)) \\
            y_+(0)
            &=C_0,
        \end{aligned}
        \right.\\
        &\lc
        \begin{aligned}
            y_-'(t)
            &=C(-1+y_-(t)) \\
            y_-(0)
            &=-C_0.
        \end{aligned}
        \right.
    \end{align*}
    By a comparison argument
    for second-order parabolic equation
    we obtain,
    \begin{equation*}
        -CT-C_0e^{CT}
        \leq
        y_-(T-t)
        \leq
        u(t,x)
        \leq
        y_+(t)
        \leq
        CT+C_0e^{CT},
    \end{equation*}
    for $(t,x)\in[0,T]\times\TT^d$.
    Therefore $u$ is uniformly bounded with respect
    to $M$.
    The desired result then stems from
    Corollary \ref{cor:exist}.
\end{proof}
In Propotitions
\ref{thm:ExiExpo} and
\ref{thm:ExiBound},
we changed the exponent
appearing in \ref{hypo:Hh}.
In the following proposition,
we assume a smaller exponent
than the one 
appearing in \ref{hypo:Hx} instead.
\begin{proposition}[\emph{Existence
    with more restrictive assumptions
    on $H_x$}]
    \label{thm:ExiLowSpaDep}
    Assume
    \ref{hypo:Hp}-\ref{hypo:mreg},
    \ref{hypo:Hh},
    \ref{hypo:Hpower},
    \ref{hypo:Hinvert},
    \ref{hypo:Hcontrac},
    \ref{hypo:Hlip},
    and the following inequality,
    \begin{equation}
        \label{eq:ExiLowSpaDep}
        \labs H_x (x,p,\mu)\rabs
        \leq
        C_0\lp 1+|p|
        +\LL_{q_0}\lp\mu\rp^{q'-1}\rp,
    \end{equation}
    for any $(x,p,\mu)\in
    \TT^d\times\RR^d\times\cP\lp\TT^d\times\RR^d\rp$.
    There exists at least
    one solution to \eqref{eq:MFGC}.
\end{proposition}

\begin{proof}
    Take $(u,m,\mu)$
    a solution to \eqref{eq:MFGCM},
    for $M\in(0,\infty)$.

    \emph{First step: we prove the following inequality,}
    \begin{equation}
        \label{eq:ExiLowSpaDep_aux1}
        \norminf[q]{\nabla_xu(t)}
        \leq 
        C\lp1+\sup_{t\leq s \leq T}\norminf{u(s)}\rp,
    \end{equation}
    for any $t\in[0,T]$,
    where $C>0$ is a constant depending only on
    the constants in the assumptions.
    We will only prove this inequality for $t=0$,
    however the proof does not use
    the additional information available at $t=0$
    (the initial condition on $m$ for example),
    so it can be repeated for any $t\in[0,T]$
    and the constant $C$ in \eqref{eq:ExiLowSpaDep_aux1}
    does not depend on $t$.

    We introduce $\vp, w^{\dd}, a,b,\dd$ and $\ee$
    as in the proof of Lemma \ref{lem:Bernstein}.
    Using \eqref{eq:ExiLowSpaDep}
    instead of \ref{hypo:Hx}, we obtain
    \begin{equation*}
        -2\vp\nabla_xu^{\dd}\cdot H_x^{\dd}\lp x,\nabla_xu,\mu\rp
        \leq
        2C_0\vp\norminf{\nabla_xu^{\dd}}
        \lp 1+
        \norminf{\nabla_xu^{\dd}}
        +\LL_{q_0}\lp\mu(t)\rp^{q'-1}\rp.
    \end{equation*}
    From this and \eqref{eq:boundL},
    one may notice
    that the right-hand side of the latter inequality
    only involves terms with exponents in
    $\norminf{w^{\dd}}$ or $\norminf{w^0}$
    nor larger than $\frac12\lp1+(q-1)(q'-1)\rp=1$.
    This and the same arguments as in the proof
    of Lemma \ref{lem:Bernstein} between \eqref{eq:boundHx}
    and \eqref{eq:ubound_PDIw},
    lead to the following inequality,
    \begin{multline*}
        \ptt w^{\dd} -\nu\Delta w^{\dd}
        + \nabla_xw^{\dd}\cdot  H_p\lp x,\nabla_x u^{\dd}, \mu\rp 
        -2\nu\frac{\vp'}{\vp} \nabla_xw^{\dd}\cdot \nabla_xu^{\dd} 
        \\
        \leq
        -C_0^{-1}
        b\norminf[-1]{u}e^{-a-b}e^{-\frac{q}2e^{-a+b}}
        \lp w^{\dd}\rp^{1+\frac{q}2}
        +b\norminf[-1]{u}e^{-a+b}
        \frac{\ll_1C_0^{q'}}{(1-\th)^{1-q'}(1-\ll_0)^{q'}}
        \norminf[1+\frac{q}2]{w^0}
        \\
        +\ee
        +C_{a,b,\th}\lp1+\norminf[-1]{u}\rp
        \lp 1+
        \norminf{w^0}\rp,
    \end{multline*}
    instead of \eqref{eq:ubound_PDIw},
    where the novelty is the exponent on $\norminf{w^0}$
    at the last line which changed from $\frac{1+q}2$
    to $1$.
    Then following the same steps as
    in the proof of Lemma \ref{lem:Bernstein}
    until the end, we obtain that,
    \begin{equation*}
        \norminf[q]{\nabla_xu}
        \leq
        C_{a,b,\th}\lp 1+\norminf{u}\rp.
    \end{equation*}
    This concludes the first step of the proof.

    \emph{Second step: obtaining a uniform estimate on $u$.}

    Using \ref{hypo:Hh},
    \eqref{eq:boundLq0q'inf} with $\th=\frac12$
    and \eqref{eq:ExiLowSpaDep_aux1},
    we obtain that,
    \begin{align*}
        \labs H\lp x,0,\mu(t)\rp\rabs
        &\leq
        C_0\lp 1
        +\LL_{q_0}\lp\mu(t)\rp^{q'}\rp
        \\
        &\leq
        2C_0
        +\frac{C_0^{q'+1}2^{q'-1}}{(1-\ll)^{q'}}
        \lp 1+\norminf[q]{\nabla_xu(t)} \rp \\
        &\leq
        C\lp1
        +\max_{t\leq s\leq T}\norminf{u(s)}\rp,
    \end{align*}
    where the constant $C$ from the previous step may have
    been increased.
    This implies that $u$ satisfies the same partial
    differential inequality as in the proof
    of Proposition \ref{thm:ExiBound},
    namely \eqref{eq:ExiBound_PDI}.
    Therefore the same arguments as in 
    Proposition \ref{thm:ExiBound} apply
    and we conclude that there exists a solution
    to \eqref{eq:MFGC}.
\end{proof}
    \begin{remark}
        Note that the exponent $q' -1$
        actually appears in several applications:
for instance,
the price impact model described
in paragraph \ref{subsec:price_impact}
in the quadratic case (i.e. $q=2$)
with $\ee=0$ (i.e. when the bidding 
and asking prices are equal),
satisfies the assumptions
in both Propositions 
\ref{thm:ExiBound} and 
\ref{thm:ExiLowSpaDep}
with an exponent exactly equal
to $q'-1$.
\end{remark}

\subsection{Existence and uniqueness results with
    a short-time horizon assumption}
\label{subsec:shorttime}
Under a short-time horizon assumption,
existence and even uniqueness of solutions
are well-known in the MFG literature.
Indeed,
when the time horizon is small,
one may obtain strong a priori estimates
under non-restrictive assumptions.
These estimates combined with Corollary \ref{cor:exist}
yield existence of solution to
\eqref{eq:MFGC} as stated in the following proposition.

\begin{proposition}[\emph{Existence with short time horizon}]
\label{thm:ExiLowT}
    Assume \ref{hypo:Hp}, 
    \ref{hypo:gregx}-\ref{hypo:Hh},
    \ref{hypo:Hx}-\ref{hypo:Hinvert},
    \ref{hypo:Hcontrac},
    and \ref{hypo:Hlip}.
    There exists $T_0>0$ such that,
    if $T\leq T_0$ then there exists
    a solution to \eqref{eq:MFGC}.
\end{proposition}

\begin{proof}
    Take $(u,m,\mu)$
    a solution to \eqref{eq:MFGCM}
    for $M\in(0,\infty)$.    
    We combine \eqref{eq:HJB2},
    \ref{hypo:Hinvert},
    \eqref{eq:boundLq0q'inf},
    \eqref{eq:Bernstein},
    and the convex inequality
    $\lp a+b\rp^{q}\leq 2^{q-1}\lp a^q+b^q\rp$,
    and we obtain
    \begin{equation}
        \label{eq:ExiLowT_PDI}        
        -\ptt u
        -\nu\Delta u
        +\nabla_xu\cdot
        \int_0^1
        H_p\lp x,s\nabla_xu(t,x),\mu(t)\rp ds
        \leq
        C\lp1+\max_{t\leq s\leq T}\norminf[q]{u(s)}\rp,
    \end{equation}
    where $C$ is a positive constant which
    depends only on the constants in the assumptions.
    We recall that $\norminf{u(T)}\leq C_0$
    by \ref{hypo:gregx}.
    Let us  consider the following
    differential equation,
    \begin{equation*}
        \lc
        \begin{aligned}
            y'(t)
            &=
            C\lp1+y^q\rp
            \\
            y(0)
            &=
            C_0.
        \end{aligned}
        \right.
    \end{equation*}
    There exists $T_0>0$ such that the latter
    differential equation admits
    a bounded solution on $[0,T_0]$.
    We suppose that $T\leq T_0$,
    then $(t,x)\mapsto y(T-t)$ is
    a super-solution to \eqref{eq:ExiLowT_PDI}.
    Hence by a comparison principle,
    we get that
    $u\leq y$.
    The same argument applies in order to 
    prove that $u\geq -y$.
    Therefore $u$ is uniformly bounded
    with respect to $M$,
    and there exists a solution to
    \eqref{eq:MFGC}
    by Corollary \ref{cor:exist}.
\end{proof}

We will now prove Theorem \ref{thm:UniLowT}
which states that uniqueness is achieved under
a short-time horizon assumption.
We believe that this uniqueness result
can be easily extended to more general Hamiltonians,
but that the short-time assumption is essential.
Indeed, numerical simulations
in \cite{achdou2020mean} show
that uniqueness does not hold
for the discrete MFGC system
obtained by approximating \eqref{eq:MFGC}
with finite differences;
we believe that uniqueness does not
hold for \eqref{eq:MFGC} either.
Theorem \ref{thm:UniLowT}
should be interpreted only as a simple example of
uniqueness result with a short-time horizon assumption.
\begin{proof}[Proof of Theorem \ref{thm:UniLowT}]
    We suppose that $T_1\leq T_0$,
    where $T_0$ was defined in Proposition \ref{thm:ExiLowT},
    so that a solution to \eqref{eq:MFGC}
    satisfies uniform estimates
    on $\norminf{u}$,
    $\|u^i\|_{C^{1,2}}$ and $\|m^i\|_{C^{0}}$
    by Lemma \ref{lem:ubound}, 
    for $i=1,2$.
    Take $(u^1,m^1,\mu^1)$ and $(u^2,m^2,\mu^2)$
    two solutions to \eqref{eq:MFGC}.
    We define $u=u^1-u^2$, $m=m^1-m^2$
    and $\aa=\aa^{\mu^1}-\aa^{\mu^2}$.

    In this proof $C>0$ is a constant
    which may differ from line to line
    and depends only on the constants in the assumptions,
    $\|u^i\|_{C^{1,2}}$ and $\|m^i\|_{C^{0}}$, 
    for $i=1,2$.

    We can repeat the proof of
    Lemma \ref{lem:distmu} replacing
    $\norminf[\bb_0]{m^1-m^2}$
    and $\norminf[\bb_0]{p^1-p^2}$
    respectively
    with $W_{q_1}\lp m^1,m^2\rp$
    and $\norminf[\bb_0]{p^1-p^2}$
    everywhere
    and we obtain that,
    \begin{equation}
        \label{eq:USTdistmu}
        \norminf{\aa(t)}
        \leq
        C\lp \norminf{\nabla_xu(t)}
        +W_{q_1}\lp m^1(t),m^2(t) \rp\rp,
    \end{equation}
    for any $t\in[0,T]$.   
    Let us consider $X^1$ and $X^2$ two random processes
    defined by
    \begin{align*}
        dX^1_t
        &=
        \aa^{\mu^1}(t,X^1)dt
        +\sqrt{2\nu}dW_t
        \\
        dX^2_t
        &=
        \aa^{\mu^2}(t,X^2)dt
        +\sqrt{2\nu}dW_t
        \\
        X^1_0
        &=
        X^2_0
        =X^0,
    \end{align*}
    where $X^0$ is a random variable on $\TT^d$ with
    law $m^0$ and $W$ is a Brownian motion independent of $X^0$.
    The respective laws of $\lp X^1_t,\aa^{\mu^1}(t,X^1_t)\rp$
    and $\lp X^2_t,\aa^{\mu^2}(t,X^2_t)\rp$ are $\mu^1(t)$ and
    $\mu^2(t)$.
    Then we obtain,
    \begin{align*}
        \EE\lb\labs X^1_t
        -X^2_t\rabs^{q_1}\rb^{\frac1{q_1}}
        &=
        \EE\lb\labs\int_0^t \aa^{\mu^1}\lp s,X^1_s\rp
        -\aa^{\mu^2}\lp s,X^2_s\rp ds\rabs^{q_1}\rb^{\frac1{q_1}}
        \\
        &\leq
        \int_0^t
        \EE\lb\labs \aa^{\mu^1}\lp s,X^1_s\rp
        -\aa^{\mu^2}\lp s,X^2_s\rp\rabs^{q_1}\rb^{\frac1{q_1}}
        ds
        \\
        &\leq
        \int_0^t
        \EE\lb\labs \aa^{\mu^1}\lp s,X^1_s\rp
        -\aa^{\mu^1}\lp s,X^2_s\rp\rabs^{q_1}\rb^{\frac1{q_1}}
        +\EE\lb\labs \aa^{\mu^1}\lp s,X^2_s\rp
        -\aa^{\mu^2}\lp s,X^2_s\rp\rabs^{q_1}\rb^{\frac1{q_1}}
        ds,
    \end{align*}
    where we used the triangle inequality for
    the $L^{q_1}$-norm twice.
    By the first additional assumption of the theorem
    and \ref{hypo:Hp},
    $\aa^{\mu^1}$ is Lipschitz continuous with respect
    to $x$ and its Lipschitz constant depends on
    $\norm{u^i}{C^{0,1}}{}$ and
    $\LL_{\infty}\lp\mu^1\rp$.
    Using the estimates from the proof of Proposition
    \ref{thm:ExiLowT},
    it only depends on the constants in the assumptions.
    This, the latter inequality
    and \eqref{eq:USTdistmu} imply
    \begin{equation*}
        \EE\lb\labs X^1_t
        -X^2_t\rabs^{q_1}\rb^{\frac1{q_1}}
        \leq
        C\int_0^t\EE\lb\labs X^1_s
        -X^2_s\rabs^{q_1}\rb^{\frac1{q_1}}ds
        +CT\sup_{0\leq t'\leq T}\lp \norminf{\nabla_xu(t')}
        +W_{q_1}\lp m^1(t'),m^2(t')\rp\rp.
    \end{equation*}
    This and Gronwall's inequality yield that,
    \begin{equation*}
        \sup_{0\leq t \leq T}
        \EE\lb\labs X^1_t
        -X^2_t\rabs^{q_1}\rb^{\frac1{q_1}}
        \leq
        CT\sup_{0\leq t\leq T}\lp \norminf{\nabla_xu(t)}
        +W_{q_1}\lp m^1(t),m^2(t)\rp\rp.
    \end{equation*}
    From now on, we assume that $T\leq \frac1{2C}$,
    so that $(1-CT)\geq \frac12$.
    Since $W_{q_1}(m^1(t),m^2(t))\leq
    \EE\lb\labs X^1_t-X^2_t\rabs^{q_1}\rb^{\frac1{q_1}}$,
    we obtain:
    \begin{equation}
        \label{eq:USTm}
        \sup_{0\leq t \leq T}
        W_{q_1}\lp m^1(t),m^2(t)\rp
        \leq
        CT\sup_{0\leq t\leq T} \norminf{\nabla_xu(t)}.
    \end{equation}
    Hence $u$ satisfies the following
    equation,
    \begin{equation*}
        \lc
        \begin{aligned}
        &-\ptt u
        -\nu\Delta u 
        =
        -H\lp x,\nabla_xu^1,\mu^1\rp
        +H\lp x,\nabla_xu^2,\mu^2\rp,
        \\
        &u(T,x)
        =
        g(x,m^1(T))
        -g(x,m^2(T)).
        \end{aligned}
        \right.
    \end{equation*}
    The right-hand side of the first line
    can be estimated
    in absolute value
    from above
    as follows:
    \begin{equation*}
        \labs
        H\lp x,\nabla_xu^1,\mu^1\rp
        -H\lp x,\nabla_xu^2,\mu^2\rp
        \rabs
        \leq
        C\sup_{0\leq t'\leq T}\norminf{\nabla_xu(t')},
    \end{equation*}
    by \ref{hypo:Hlip},
    \eqref{eq:USTdistmu}
    and \eqref{eq:USTm}.
    Since $u(T,\cdot)\in C^{1+\bb}\lp\TT^d\rp$,
    Theorem $6.48$ in \cite{MR1465184} yields that 
    $u\in C^{\frac12+\frac{\beta}2,1+\beta}\lp[0,T]\times\TT^d\rp$
    and it satisfies:
    \begin{multline*}
        \sup_{t\in[0,T]}
        \norminf{\nabla_xu(t)}
        \leq
        \norminf{\nabla_x\lp 
        g(\cdot,m^1(T))
        -g(\cdot,m^2(T))\rp}
        \\
        +CT^{\frac{\bb}2}
        \sup_{t\in[0,T]}
        \lp\norminf{\nabla_xu(t)}
        +\norm{g(\cdot,m^1(T))
        -g(\cdot,m^2(T))}{C^{1+\bb}}{}\rp.
    \end{multline*}
    This, \eqref{eq:gregm} and \eqref{eq:USTm}
    yield,
    \begin{equation*}
        \sup_{t\in[0,T]}
        \norminf{\nabla_xu(t)}
        \leq
        CT^{\frac{\bb}2}\sup_{t\in[0,T]}
        \norminf{\nabla_xu(t)}.
    \end{equation*}
    Thus if we suppose furthermore that $T<C^{-\frac2{\bb}}$,
    then $\nabla_xu=0$,
    so $m=0$ by \eqref{eq:USTm},
    then $\mu^1=\mu^2$ by \eqref{eq:USTdistmu},
    and finally $u^1$ and $u^2$ solve
    the same Hamilton-Jacobi-Bellman equation
    with the same terminal condition,
    so by uniqueness $u=0$.
    
    Therefore, we proved the uniqueness for $T<T_1$
    where $T_1$ is defined by
    $T_1=\min T_0,\lp C^{-\frac2{\bb}},C^{-1}\rp$.
\end{proof}

\section{Applications}
\label{sec:appli}
Here, we are going to work on $\TT^d$,
while it would be more realistic to work
in the whole space $\RR^d$
for the applications considered below.
We would like to recall that
the existence results
contained in the present work hold for
MFGC systems on $\RR^d$
using the method
introduced in \cite{MonoMFGC}
to pass from the torus to the
whole Euclidean space.
Therefore, the conclusions
of this section may be adapted
to treat the same applications on $\RR^d$.

\subsection{Exhaustible ressource model with nonpositively correlated ressources}
\label{subsec:exhau}
This model is often referred to
as Bertrand and Cournot competition model for exhaustible ressources,
introduced in the independent works
of Cournot \cite{cournot1838recherches} and Bertrand \cite{Bertrand};
its mean field game version in dimension one was introduced in
\cite{MR2762362} and numerically analyzed in
\cite{MR3359708};
for theoretical results see
\cite{2019arXiv190205461F,MR3755719,MR4064472,MR3888969}.
We consider a continuum of producers selling exhaustible ressources.
The production of a representative agent is $(q_t)_{t\in[0,T]}$;
the agents differ in their production capacities
$X_t\in\TT$ (the state variable),
that satifies,
\begin{equation*}
    dX_t
    =
    -q_tdt
    +\sqrt{2\nu} dW_t,
\end{equation*}
where $\nu>0$ and $W$ is a Brownian motion.
Each producer is selling a different ressource
and has her own consumers.
However, the ressources are substitutable
and any consumer may change her mind
and buy from a competitor
depending on the degree of competition
in the game (which is characterized by $\ee$ in the
linear demand case below for instance).
Therefore, the selling price per unit of ressource
that a producer can make when
she sales $q$ units of ressource,
depends naturally on $q$
and on the quantity produced by the other
agents.
The price satisfies a
supply-demand relationship,
and is given by
$P\lp q,\qo\rp$,
where $\qo$ is the aggregate demand
which depends on 
the overall distribution of productions
of the agents.
A producer tries to 
maximize her profit,
or equivalently to minimize
the following quantity,
\begin{equation*}
    \EE\lb\int_0^T -P(q_t,\qo_t)\cdot q_tdt
    +g\lp X_T\rp\rb,
\end{equation*}
where $g$ is a terminal cost
which often penalizes the producers
who have non-zero production capacities at the end of the game.
In the Cournot competition, see \cite{cournot1838recherches},
a producer is controling her production $q$.
Like the
MFG version of the Bertrand and Cournot competition
introduced
in \cite{MR3359708},
here we consider the Bertrand formulation \cite{Bertrand},
where an agent directly controls her selling price
$\aa=P(q,\qo)$.
After inverting the latter equality,
the production can be viewed as a function
of the price and the mean field.
Mathematically this corresponds to writing
$q=Q\lp \aa,\aao\rp$.

In \cite{MR3359708},
the authors considered a linear demand system depending on
$\qo_{\text{lin}}=\int_{\TT}q(x)dm(x)$,
and a price satisfying
$\aa=P_{\text{lin}}(q,\qo_{\text{lin}})=1-q-\ee \qo_{\text{lin}}$.
In this case, the running cost $L^{\text{lin}}$
and its Legendre transform $H^{\text{lin}}$
are defined by
\begin{align*}
    L^{\text{lin}}\lp \aa,\mu\rp
    &=
    \aa^2
    +\frac{\ee}{1+\ee}\aa\aao
    -\frac1{1+\ee}\aa,
    \\
    H^{\text{lin}}\lp p,\mu\rp
    &=
    \frac14\lp p+\frac{\ee}{1+\ee}\aao-\frac1{1+\ee}\rp^2,
\end{align*}
where $\aa,p\in\RR$, $\mu\in\cP\lp\TT\times\RR\rp$
and $\aao$ is defined by
$\aao=\int_{\TT\times\RR}\aat d\mu(y,\aat)$.
Therefore the system of MFGC has the following form,
\begin{equation}
    \label{eq:exhaulin}
    \lc
    \begin{aligned}
        &-\ptt u
        -\nu\Delta u
        +
        \frac14\lp \nabla_xu+\frac{\ee}{1+\ee}\aao-\frac1{1+\ee}\rp^2
        =
        0,
        \\
        &\ptt m
        -\nu\Delta m
        -\divo\lp\frac12\lp \nabla_xu+\frac{\ee}{1+\ee}\aao-\frac1{1+\ee}\rp m\rp
        =
        0,
        \\
        &\aao(t)
        =
        -\int_{\TT}
        \frac12\lp \nabla_xu+\frac{\ee}{1+\ee}\aao(t)-\frac1{1+\ee}\rp
        dm(t,x),
        \\
        &u(T,x)
        =
        g(x),
        \\
        &m(0,x)
        =
        m_0(x),
    \end{aligned}
    \right.
\end{equation}
for $(t,x)\in[0,T]\times\TT$.
Roughly speaking, $\ee=0$ corresponds
to a monopoly in which a producer does not
suffer from competition,
and she plays as if she was alone in the game.
Conversely, $\ee=\infty$ stands for all the producers
selling the same ressource and the consumers not having
any preference.

Here, Theorem \ref{thm:MFGCexi} \ref{exisub:Hx}
implies the following existence result.
\begin{proposition}
    If $m_0$ and $g$ satisfy \ref{hypo:gregx}
    and \ref{hypo:mreg},
    there exists a solution to \eqref{eq:exhaulin}
    for any $\ee\in(0,\infty)$.
\end{proposition}
To prove it,
we may take
$q=2$, $q_0=1$,
$\ll_0=\frac{\ee}{2(1+\ee)}$,
$\ll_1=1$,
and $C_0=\frac12$
in \ref{hypo:Hinvert};
then we check the assumptions
of Theorem \ref{thm:MFGCexi} \ref{exisub:Hx}.
In this case, the inequality in \ref{hypo:Hpower}
has the form $1<\lp\frac{2+\ee}{1+\ee}\rp^2$,
and is satisfied for any $\ee\in(0,\infty)$.

Here, the Lagrangian $L^{\text{lin}}$
satisfies a monotonicity
assumption, but the latter existence result
does not take advantage of it.
We refer to \cite{MonoMFGC} for
a uniqueness result and an other
existence result for the solution
to \eqref{eq:exhaulin} using this
monotonicity assumption.
Generalizations
of \eqref{eq:exhaulin} to larger
dimensions with more general
Hamiltonians and prices are also
discussed in \cite{MonoMFGC}
under the monotonicty assumption.

In what follows, we provide a simple
example of a generalization
of \eqref{eq:exhaulin} in which
the monotonicity assumption
does not hold and the results
in \cite{MonoMFGC} do not apply anymore.
However, the results in the present work may
hold in some cases
even without
the monotonicity assumption.

Let us consider a model in which every
producer sells $d$ different kinds of ressources.
The price of each ressource depends on the mean field
like in \eqref{eq:exhaulin}.
Namely,
we take $Q=M\aao-\aa$
which is now a $d$-dimensional vector
and where $M\in\RR^{d\times d}$ is a given matrix.
This leads to the following MFGC system,
\begin{equation}
    \label{eq:exhauNC}
    \lc
    \begin{aligned}
        &-\ptt u
        -\nu\Delta u
        +
        \frac14\lp \nabla_xu+M\aao\rp^2
        =
        f(x,m),
        \\
        &\ptt m
        -\nu\Delta m
        -\divo\lp\frac12\lp\nabla_xu+M\aao\rp m\rp
        =
        0,
        \\
        &\aao(t)
        =
        -\lp I_d+\frac12M\rp^{-1}
        \int_{\RR^d}
        \nabla_xu(t,x)
        dm(t,x),
        \\
        &u(T,x)
        =
        g(x),
        \\
        &m(0,x)
        =
        m_0(x),
    \end{aligned}
    \right.
\end{equation}

\begin{proposition}
    Assume
    \ref{hypo:gregx},
    \ref{hypo:mreg},
    that $M$ has an operator norm smaller than $1$,
    and that $f$ is 
    continuous,
    and differentiable with respect to $x$
    with continuous derivatives.
    There exists a solution to \eqref{eq:exhau}.
\end{proposition}
The proof consists in taking
$q=2$,
$q_0=1$,
$\ll_1=1$,
$C_0=\frac12$
in \ref{hypo:Hinvert},
and
$\ll_0=\frac{\norm{M}{}}2$,
where $\norm{M}{}$
is the operator norm of $M$;
and we check the assumptions
of Theorem \ref{thm:MFGCexi} \ref{exisub:Hx}.

The monotonicity assumption
discussed in \cite{MonoMFGC}
is equivalent to assuming
that $M$ is a positive semi-definite matrix.
Here, we do not make such an assumpion.

What we have in mind in the latter example
is the case where the prices of the different
ressources may be negatively correlated,
like cars and oil
(if the production of cars increases,
then the demand for oil also increases
and  the price of oil rises
while the price of cars decreases),
or pesticides and medicines,
or gold and other raw materiels.
To our knowledge,
such a generalization of the exhaustible ressource model
to negatively correlated ressources
is new in the MFG literature.

More generally, we believe
that our results hold for the following
MFGC system under various different 
sets of assumptions that we will not
detailed here,
\begin{equation}
    \label{eq:exhau}
    \lc
    \begin{aligned}
        &-\ptt u
        -\nu\Delta u
        +
        H\lp x,\nabla_xu+Q(t,x,\mu)\rp
        =
        f(t,x,m(t)),
        \\
        &\ptt m
        -\nu\Delta m
        -\divo\lp H_p\lp  x,\nabla_xu+Q(t,x,\mu)\rp m\rp
        =
        0,
        \\
        &\mu(t)
        =
        \Bigl(
        I_d,
        -H_p\lp \cdot,\nabla_xu(t,\cdot)+Q(t,\cdot,\mu(t))\rp\#m(t)
        \\
        &u(T,x)
        =
        g(x,m(T)),
        \\
        &m(0,x)
        =
        m_0(x),
    \end{aligned}
    \right.
\end{equation}
where $Q:[0,T]\times\TT^d\times\cP\lp\TT^d\times\RR^d\rp
\to \RR^d$ is a vector characterizing the mean field
interactions.

\subsection{Price impact models with bid and ask prices}
\label{subsec:price_impact}
The price impact model without bid and ask prices
is inspired by the Almgren and Chriss's model \cite{Almgren2000OptimalEO},
and was introduced in the MFG literature in \cite{MR3805247}
and \cite{MR3325272}
where existence and uniqueness results are proved
when the admissible controls stay in a compact set.
Here we consider an extension with bid and ask prices.

We suppose that a continuum of agents are trading an asset,
the state of a representative agent is $X_t$ the
amount of this asset she owns.
Her control $\aa$ is the quantity she buys (if $\aa\geq 0$)
or sell (if $\aa<0$).
The state space is the one-dimensional torus $\TT$,
and $X_t$ is given by,
\begin{equation*}
    dX_t
    =
    \aa_tdt
    +\sigma dW_t,
\end{equation*}
where $W$ is a Brownian motion, and $\sigma>0$ is a real constant.
We define $S_t$ as the asking price of the asset,
and $\ee\lp \mu(t)\rp$ as the difference between the bidding and
asking prices, where $\mu(t)$ is the law
of $(X_t,\aa_t)$.
The agent buys at the bidding price $S_t+\ee\lp\mu_t\rp$,
thus her cash is given by
\begin{equation*}
    dK_t
    =
    -\lp \aa_t S_t
    +\aa_t \ee\lp\mu(t)\rp+\ell(\aa_t)\rp dt,
\end{equation*}
where $\ell$ is a differentiable function standing for the transaction cost.
The price $S_t$ evolves accordingly with the amount of transactions at time $t$,
it satisfies the following SDE,
\begin{align*}
    &dS_t
    =
    A\lp\mu(t)\rp dt,
    \\
    &\text{where }
    A\lp\mu(t)\rp
    =
    \int_{\TT\times\RR}
    \ell'(\aa)d\mu(t,x,\aa),
\end{align*}
The wealth of a representative agent is given by
$V_t=V_0+X_tS_t+K_t$ and it satisfies the following SDE,
\begin{equation}
    dV_t
    =
    \lp X_tA\lp \mu(t)\rp
    -\ell\lp\aa_t\rp
    -\ee\lp \mu(t)\rp\aa_t\rp dt
    +\sigma S_tdW_t.
\end{equation}
The objective function that she will try to maximize is given by,
\begin{equation*}
    \EE\lb V_T
    -\int_0^T f(X_t)dt
    -g(X_T)\rb,
\end{equation*}
where $f$ and $g$ are penalization costs for holding stocks.
Here, the Lagrangian and Hamiltonian are given by,
\begin{align*}
    L^{\text{PI}}\lp x,\aa,\mu\rp
    &=
    \ell(\aa)+\aa\ee\lp\mu\rp
    -xA\lp\mu\rp,
    \\
    H^{\text{PI}}\lp x,p,\mu\rp
    &=
    h\lp p+\ee\lp\mu\rp\rp
    +xA\lp\mu\rp,
\end{align*}
for $\lp x,\aa,\mu\rp\in\TT\times\RR\times\cP\lp\TT\times\RR\rp$,
where $h$ is the Legendre transform of $\ell$.

The linear-quadratic case with $\ee=0$
is treated in \cite{MR3752669}.
Here, taking $\ee=0$ corresponds
to assuming that the bidding and asking prices coincide.
In this case the optimal control is given by
$-h_p(p)$ and does not depend explicitely on $\mu$.
If $\ee\neq0$,
the optimal control depends explicitely on $\mu$
and $L^{\text{PI}}$ is not separable in $\aa$ and $\mu$,
this prevents us from using the results in 
\cite{MR3752669}.

Let us give an example of choices for
the functions $\ell$ and $\ee$
under which our result apply and a solution
of the MFGC price impact model exists.
\begin{proposition}
    Assume \ref{hypo:gregx},
    \ref{hypo:mreg},
    that $f$ is $C^1$,
    and that $c$ and $\ee$
    are respectively given by
    $\ell(\aa)=\frac{|\aa|^2}2$
    and $\ee\lp\mu\rp
    =\eet\lp\int_{\TT\times\RR}|\aa|^2d\mu\lp x,\aa\rp\rp^{\frac12}$,
    where $0<\eet<\frac12$.
    There exists a solution to \eqref{eq:MFGC}
    with $H^{\text{PI}}$.
\end{proposition}
This existence result is a consequence of
\ref{thm:MFGCexi} \ref{exisub:H0},
where the assumptions are satisfied for
$q=q_0=2$,
$\ll_0=\ee$,
$\ll_1=\frac14$
and $C_0=1$ in
\ref{hypo:Hinvert}.
We would like to insist on the fact that
Theorem \ref{thm:MFGCexi} \ref{exisub:H0}
provides the existence of solutions
for a wild class of Hamiltonian,
larger than the one of the latter
proposition and which goes beyond
the linear-quadratic case.

Let us mention that we would be interested in defining
the bidding price by $(1+\eet)S_t$,
where $\eet>0$.
The associated MFGC system cannot be 
using the conclusions of the present work
because the mean field interaction at time $t$
would depend not only on $\mu_t$
but on $\lp\mu_s\rp_{s\in[0,t]}$.
However,
we believe that existence holds
under similar assumptions as here, and
we plan to prove it in forthcoming works.

\subsection{First-order flocking model with velocity as controls}
\label{subsec:flocking}
Cucker and Smale proposed a form of Vicseck model in
\cite{MR2324245} to illustrate the behavior of flocks of birds.
This model is of second-order in the sense that
the state of an agent is given by a
couple $(x,v)$ standing for her position and
velocity respectively, and the equation of evolution
of her state involves considering her acceleration.

A game version of this model in which an agent controls her
acceleration has been introduced in \cite{5706992},
the authors derived a MFG formulation in the infinite horizon case.
Here we are interested in the finite horizon problem
which was studied in \cite{MR3325272,MR3752669}.
This model is still of second-order.
More precisely the state of an agent is
given by $(X_t,V_t)_{t\in[0,T]}$  respectively
her position and velocity,
two random processes which satisfy the following system of
stochastic differential equations,
\begin{align*}
    dX_t
    &=
    V_tdt,
    \\
    dV_t
    &=
    a_tdt
    +\sigma dW_t,
\end{align*}
where $a_t$ is the individual's acceleration vector
and her control,
$W$ is a $d$-dimentional Brownian motion,
and $\sigma\in \RR^{d\times d}$ is a positive definite matrix.
The cost that a representative agent tries to minimize is given by
\begin{equation*}
    \EE\lb\int_0^T
    \frac{\labs a_t\rabs^2}2
    +\frac12\labs\int_{\TT^d}\lp v-V^i_t\rp
    \vp\lp\labs x -X^i_t\rabs\rp d\mu(t,x,v)\rabs^2
    +f(X_t)dt\rb,
\end{equation*}
where $\mu(t)\in\cP\lp\TT^d\times\RR^d\rp$ is the joint distribution
of states and velocities of the agents,
$\vp$ is a $C^1$ nonincreasing function, and $f$ is a $C^1$
function modeling the spatial preferences of the agents
(for instance, we can take $f$ significantly
smaller in some areas which corresponds to where the food is).

Here we consider an alternative viewpoint in which an agent directly controls
her velocity.
This is a first-order model since the state of an agent is now
given by a vector of $\TT^d$, and the acceleration does not appear
anymore in the dynamics of a given agent,
which is given by
\begin{equation*}
    dX_t
    =
    \aa_tdt
    +\sigma dW_t.
\end{equation*}
Here, the cost that an agent tries to minimize is given by
\begin{equation*}
    \EE\lb\int_0^T
    \frac{\labs\aa_t\rabs^2}2
    +\frac12\labs\int_{\TT^d\times\RR^d}\lp \aat-\aa_t\rp
    \vp\lp\labs x -X^i_t\rabs\rp d\mu(t,x,\aat)\rabs^2
    +f\lp X_t\rp dt
    \rb.
\end{equation*}
First-order physical models are generally easier to study
than second-order models.
However the price we paid here
to go from a second-order model to a first-order model
is to consider a MFGC system instead of a MFG system
without interaction through the controls.

If $\mu\in\cP\lp\TT^d\times\RR^d\rp$ and $m\in\cP\lp\TT^d\rp$ are
such that $m$ is the marginal of $\mu$ with respect to $\TT^d$,
we define $A(x,\mu)$ and $Z(x,\mu)$ by,
\begin{equation*}
    \lc
    \begin{aligned}
        A(x,\mu)
        &=
        \int_{\TT^d\times\RR^d}
        \aat\vp\lp\labs x-y\rabs\rp d\mu(y,\aat),
        \\
        Z(x,\mu)
        &=
        \int_{\TT^d}
        \vp\lp\labs x-y\rabs\rp dm(y),
    \end{aligned}
    \right.
\end{equation*}
for $x\in\TT^d$.
We define the Lagrangian of the first-order flocking model by,
\begin{equation*}
    L^{\text{FM}}\lp x,\aa,\mu\rp
    =
    \frac{\labs\aa\rabs^2}2
    +\frac12\labs Z(x,\mu)\aa-A(x,\mu)\rabs^2
    +f(x),
\end{equation*}
for $\lp x,\aa,\mu\rp\in\TT^d\times\RR^d\times\cP\lp\TT^d\times\RR^d\rp$,
and the Hamiltonian by,
\begin{equation*}
    H^{\text{FM}}(x,p,\mu)
    =
    \frac1{2\lp 1+Z(x,\mu)^2\rp}
    \lp \labs p\rabs^2
    -2Z(x,\mu)A(x,\mu)\cdot p
    -\labs A(x,\mu)\rabs^2\rp
    -f(x),
\end{equation*}
for $p\in\RR^d$, such that $H^{\text{FM}}$
is the Legendre's transform of $L^{\text{FM}}$.

\begin{proposition}
    Under assumptions \ref{hypo:gregx}
    and \ref{hypo:mreg},
    there exists $T_0>0$
    such that if $T<T_0$,
    there exists a unique solution
    to \eqref{eq:MFGC}
    with $H^{\text{FM}}$.
\end{proposition}
Hereafter, we present an other
model for crowd motion
which is very similar
to the first-order flocking model
discussed above.
The main difference between these two models
is the normalization constants.
However, the assumptions and conclusions
of this work are more adapted
to the following crowd motion
model and we can derive more existence
results for it.
We believe that these results can
be adapted to the first-order
Cucker-Smale system.

\subsection{A model of crowd motion}
\label{subsec:crowd}
This model of crowd motion has been numerically studied
in \cite{achdou2020mean} in the quadratic case,
and has some similarities with the
first-order flocking model presented in the previous paragraph.
For $(x,\mu)\in\TT^d\times\cP\lp\TT^d\times\RR^d\rp$,
we define $V(x,\mu)$ and $Z_{q_0}(x,\mu)$ by
\begin{equation*}
    \lc
    \begin{aligned}
        V(x,\mu)
        &=
        \frac1{Z_{q_0}(x,\mu)}
        \int_{\TT^d\times\RR^d}
        \aat k(x,y) d\mu(y,\aat),
        \\
        Z_{q_0}(x,\mu)
        &=
        \lp\int_{\TT^d}
        k(x,y)^{q_0'}dm(y)\rp^{\frac1{q_0'}},
    \end{aligned}
    \right.
\end{equation*}
where $q_0\in(1,\infty]$,
$q_0'$ is the conjugate exponent of $q_0$,
$k:\TT^d\times\TT^d\rightarrow\RR_+$ is a nonnegative $C^1$ kernel,
and $m\in\cP\lp\TT^d\rp$ is the marginal 
of $\mu$ with respect to $\TT^d$.
The quantity $V(x,\mu)$ is called the average drift.

The state of a representative agent is given by her position $X_t\in\TT^d$
and she controls her velocity $\aa_t$,
\begin{equation*}
    dX_t
    =
    \aa_t dt
    +\sqrt{2\nu}dW_t.
\end{equation*}
Her objective is to minimize the cost given by,
\begin{equation*}
    \EE\lb\int_0^T
    \frac{\th}{a'}\labs \aa_t-\llt V(X_t,\mu(t))\rabs^{a'}
    +\frac{1-\th}{b'}\labs \aa_t\rabs^{b'}
    +f(X_t)dt
    +g(X_T)\rb,
\end{equation*}
where $-1<\llt<1$ and $0\leq\th\leq1$ are two constants 
standing for the preference of an individual to have a similar (resp. opposite)
control as the mainstream when $\llt>0$ (resp. $\llt<0$),
$f$ and $g$ are respectively the running cost
and the terminal cost which encode the
spatial preferences of the agents,
and $a',b'> 1$ are exponents.

Here, we take $q=\min(a,b)$. 
In this model we define the Lagrangian by,
\begin{equation}
    \label{eq:defLcrowd}
    L\lp x,\aa,\mu\rp
    =
    \frac{\th}{a'}\labs\aa-\llt V(x,\mu)\rabs^{a'}
    +\frac{1-\th}{b'}\labs\aa\rabs^{b'},
\end{equation}
and the Hamiltonian as its Legendre transform.
If $a=b=2$, $H$ is given by
    \begin{equation*}
        H(x,p,\mu)
        =
        \frac{|p|^2}2
        -\llt\th p \cdot V(x,\mu)
        -\frac{\llt^2\th(1-\th)}{2}|V(x,\mu)|^2.
    \end{equation*}
If $\th=1$, $H$ satisfies
    \begin{equation*}
        H(x,p,\mu)
        =
        \frac{1}{a}|p|^{a}
        -\llt p\cdot V(x,\mu).
    \end{equation*}
For other choices of the parameters $a$,$b$ and $\th$,
$H$ does not admit an explicit form.

\begin{proposition}
    \label{prop:exicrowd}
    Assume that $g$ and $m_0$ satisfy
    \ref{hypo:gregx} and \ref{hypo:mreg}
    respectively.
    There exists a solution to 
    \eqref{eq:MFGC} where $H$
    is the Legendre transform of
    $L$ given in \eqref{eq:defLcrowd},
    under one of the following
    assertions,
    \begin{enumerate}[a)]
        \item
            \label{exicrowd:ab}
            $q_0\leq q'$
            and $a\neq b$,
        \item
            \label{exicrowd:lowparam}
            $q_0\leq q'$
            and one of the following assertions
            is satisfied,
            \begin{enumerate}[\roman*)]
                \item
                    $\th<\th_0$,
                \item
                    $\th>1-\th_0$,
                \item
                    $\labs\llt\rabs<\ll_0$,
            \end{enumerate}
            where $\th_0,\ll_0\in(0,1)$ are constants
            coming from Theorem \ref{thm:MFGCexi} \ref{exisub:H0},
        \item
            \label{exicrowd:th1}
            $\th=1$,
        \item
            \label{exicrowd:k}
            $k(x,y)$ is constant,
        \item
            \label{exicrowd:T}
            $T<T_0$,
            where $T_0$ is a positive constant 
            coming from Theorem \ref{thm:MFGCexi} \ref{exisub:T}.
    \end{enumerate}
\end{proposition}

\begin{proof}
    We refer to the appendix,
    Lemma \ref{lem:checkcrowd}
    for the proof that $H$
    satisfies
    \ref{hypo:Hp}-\ref{hypo:mreg},
    \ref{hypo:Hh}-\ref{hypo:Hpower},
    \ref{hypo:Hinvert}-\ref{hypo:Hcontrac},
    and \ref{hypo:Hlip}.
    The existence results
    \ref{exicrowd:th1},
    \ref{exicrowd:k} and 
    \ref{exicrowd:T}
    are direct consequences
    of Theorem \ref{thm:MFGCexi} \ref{exisub:H0},
    \ref{exisub:Hx} and
    \ref{exisub:T}
    respectively.

We define $\Lt(\aa,V)$ by
\begin{equation*}
    \Lt(\aa,V)
    =
    \frac{\th}{a'}\labs\aa-\llt V\rabs^{a'}
    +\frac{1-\th}{b'}\labs\aa\rabs^{b'},
\end{equation*}
for $\aa,V\in\RR^d$,
$\Ht(p,V)$ as the Legendre transform of $\Lt$
with respect to its first argument,
and $\aaa(p,V)$ as the unique control
which achieves the maximum in the definition
of $\Ht$ (it is unique because $\Lt$ is strictly convex with respect to $\aa$).

    \emph{Proof of \ref{exicrowd:ab}.}
    Take $V\in\RR^d$
    and $\aaa=\aaa(0, V)$,
    since $\aaa$ achieves the maximum
    in the definition of $\Ht(0,V)$,
    we know that
    \begin{equation*}
        0
        =
        \th\labs \aaa-\llt V\rabs^{a'-2}
        (\aaa-\llt V)
        +(1-\th)\labs \aaa\rabs^{b'-2}\aaa,
    \end{equation*}
    which implies
    \begin{equation}
        \label{eq:estaa}
        \th\labs \aaa-\llt V\rabs^{a'-1}
        =
        (1-\th)\labs \aaa\rabs^{b'-1},
    \end{equation}
    and then
    \begin{equation}
        \label{eq:estaa2}
        \lp
        \th^{a-1}(1-\th)^{2-a}
        |\aaa|^{\frac{(a'-2)(b'-1)}{a'-1}}
        + (1-\th)|\aaa|^{b'-2}
        \rp
        \aaa
        =
        \llt\th^{a-1}(1-\th)^{2-a}
        |\aaa|^{\frac{(a'-2)(b'-1)}{a'-1}}
        V.
    \end{equation}
    The two latter equalities yield 
    $\displaystyle{\lim_{V\rightarrow+\infty}
    |\aaa(0,V)|=+ \infty}$.
    We make out two cases:

    \begin{itemize}
        \item
            if $a>b$ then we have
            $\frac{(a'-2)(b'-1)}{a'-1}<b'-2$,
            and $|\aaa|=\displaystyle{\underset{+\infty}o(|V|)}$.
            Therefore, \eqref{eq:estaa2} yields
            \begin{equation*}
                |\aaa|^{b'-1-\frac{(a'-2)(b'-1)}{a'-1}}
                =
                \underset{+\infty}O(|V|),
            \end{equation*}
            and $b'-1-\frac{(a'-2)(b'-1)}{a'-1}
            =\frac{a-1}{b-1}>1$,
            so we obtain
            \begin{equation*}
                |\aaa|
                =
                \underset{+\infty}O
                \lp |V|^{\frac{b-1}{a-1}}\rp,
            \end{equation*}
            which yields
            \begin{equation*}
                \Ht(0,V)
                =
                \underset{+\infty}O
                \lp |V|^{a'}\rp
                +\underset{+\infty}O
                \lp |V|^{\frac{b-1}{a-1}b'}\rp,
            \end{equation*}
            with $a'<b'$,
            and $\frac{b-1}{a-1}b'<b'$,
            and $b=q$.
            
        \item if $a<b$ then we have
            $\frac{(a'-2)(b'-1)}{a'-1}>b'-2$,
            and $\aaa=\llt V+\displaystyle{\underset{+\infty}o(|V|)}$.
            Therefore, \eqref{eq:estaa2} yields
            \begin{equation*}
                \lp 1
                +\underset{+\infty}O
                \lp |V|^{b'-2-\frac{(a'-2)(b'-1)}{a'-1}}\rp
                \rp\aaa
                =
                \llt V.
            \end{equation*}
            We notice that
            $b'-2-\frac{(a'-2)(b'-1)}{a'-1}
            =
            \frac{b'-a'}{a'-1}<0$,
            and we obtain
            \begin{equation*}
                \aaa
                =
                \llt
                V
                +\underset{+\infty}O
                \lp |V|^{1+\frac{b'-a'}{a'-1}} \rp
                =
                \llt
                V
                +\underset{+\infty}O
                \lp |V|^{\frac{a-1}{b-1}} \rp.
            \end{equation*}
            This implies
            \begin{equation*}
                \Ht(0,V)
                =
                \underset{+\infty}O
                \lp |V|^{\frac{b'-1}{a'-1}a'} \rp,
                +\underset{+\infty}O
                \lp |V|^{b'} \rp,
            \end{equation*}
            with $b'<a'$,
            and $\frac{a-1}{b-1}a'<a'$,
            and $a=q$.
    \end{itemize}
    We conclude by \eqref{eq:Vbound} and
    Theorem \ref{thm:MFGCexi} \ref{exisub:H0q}. 

    \emph{Proof of \ref{exicrowd:lowparam}}

    Here, we assume that $a=b$
    since the case $a\neq b$
    is addressed in \ref{exicrowd:ab}.

    Take $V\in \RR^d$,
    and $\aaa=\aaa(0,V)$.
    In this case,
    $\Ht(0,V)$ admits an explicit form
    given by
    \begin{equation*}
        \Ht(0,V)
        =
        -\frac{\labs\llt\rabs^{a'}}{a'}
        \frac{\th(1-\th)^{a}+(1-\th)\th^{a}}
        {\lp
        (1-\th)^{a-1}
        +\th^{a-1}
        \rp^{a'}}
        |V|^{a'}.
    \end{equation*}
    Therefore,
    taking $\llt$, $\th$
    or $(1-\th)$ small enough
    allows one to conclude by
    \eqref{eq:Vbound} and
    Theorem \ref{thm:MFGCexi} \ref{exisub:ll}.

\end{proof}

{\bf Acknowledgements.}
I wish to express my gratitude to Y. Achdou
and P. Cardaliaguet for technical advices,
insightful comments and corrections.
The work was supported by the ANR project
MFG ANR-16-CE40-0015-01.

\bibliographystyle{plain}
\bibliography{MFGbiblio}

\begin{appendix}

\section{Verification of the assumptions
    for the model of crowd motion}
We start by establishing some properties of the function
$V$ in the following lemma.
\begin{lemma}
    \label{lem:GGinf}
    The function $V$ is $C^{1}$ with respect
    to $x$ and it satisfies
            \begin{equation}
                \label{eq:Vbound}
                \norminf{V\lp\cdot,\mu\rp}
                \leq
                \LL_{q_0}\lp\mu\rp,
            \end{equation}
            where $\mu\in\cP\lp\TT^d\times\RR^d\rp$.

            For $m\in\cP \lp \TT^d\rp$
            and  $\mu^1,\mu^2\in\cP_m\lp\TT^d\times\RR^d\rp$,
            the following inequality is satisfied,
            \begin{equation}
                \label{eq:Vmdiff}
                \norminf{V\lp\cdot,\mu^1\rp
                -V\lp\cdot,\mu^2\rp}
                \leq
                \norm{\aa^{\mu^1}-\aa^{\mu^2}}{L^{q_0}\lp m\rp}.
            \end{equation}

            For $R>0$,
            there exists $C_R>0$ a constant
            such that,
            \begin{equation}
                \label{eq:Vmudiff}
                \norminf{V\lp\cdot,\mu^1\rp
                -V\lp\cdot,\mu^2\rp}
                \leq
                C_R\lp
                \norminf{\aa^{\mu^1}-\aa^{\mu^2}}
                +\norminf{m^1-m^2}
                \rp,
            \end{equation}
            for $\lp m^i,\mu^i\rp$
            such that
            $m^i\in\cP\lp\TT^d\rp\cap C^0\lp\TT^d\rp$
            with $m^i\geq R^{-1}$,
            $\mu^i\in\cP\lp\TT^d\times\RR^d\rp$
            with $\aa^{\mu^i}\in C^0\lp\TT^d\times\RR^d\rp$
            and $\norminf{\aa^{\mu^i}}\leq R$,
            $i=1,2$.
\end{lemma}

\begin{proof}
    The function $V$ has at least the same regularity as $k$
    with respect to the state variable since $V$ is the convolution
    product of $k$ with a probability measure.
    Then \eqref{eq:Vbound} and \eqref{eq:Vmdiff} are straightforward
    using Hölder inequality.
    Let us take the same notation as in  \eqref{eq:Vmudiff},
    for $x\in\TT^d$ we get
    \begin{align*}
        \labs V\lp x,\mu^1\rp-V\lp x,\mu^2\rp\rabs
        &\!\begin{multlined}[t][10.5cm]
            \leq
            \frac1{Z_{q_0}\lp x,\mu^1\rp}
            \int_{\TT^d} k(x,y)\labs\aa^{\mu^1}(y)-\aa^{\mu^2}(y)\rabs dm^1(y)
            \\
            +\frac1{Z_{q_0}\lp x,\mu^1\rp}
            \int_{\TT^d} k(x,y)\labs\aa^{\mu^2}(y)\rabs \labs m^1(y)-m^2(y)\rabs dy
            \\
            +\labs\frac1{Z_{q_0}\lp x,\mu^1\rp}
            -\frac1{Z_{q_0}\lp x,\mu^2\rp}\rabs
            \int_{\TT^d} k(x,y)\labs\aa^{\mu^2}(y)\rabs dm^2(y)
        \end{multlined}
        \\
        &\!\begin{multlined}[t][10.5cm]
            \leq
            \norminf{\aa^{\mu^1}-\aa^{\mu^2}}
            +\frac1{Z_{q_0}\lp x,\mu^1\rp}
            \int_{\TT^d}k(x,y)dy
            \norminf{\aa^{\mu^2}}
            \norminf{m^1-m^2}
            \\
            +\frac1{Z_{q_0}\lp x,\mu^1\rp}
            \norm{\aa^{\mu^2}}{L^{q_0}\lp m^2\rp}\labs 
            Z_{q_0}\lp x,\mu^1\rp
            -Z_{q_0}\lp x,\mu^2\rp
            \rabs.
        \end{multlined}
    \end{align*}
    Moreover,
    we know that
    $Z_{q_0}\lp x,\mu^1\rp\geq R^{-\frac1{q_0'}}
    \lp\int_{\TT^d}k(0,y)^{q_0'}dy\rp^{\frac1{q_0'}}>0$
    where the right-hand side does not depend on $x$,
    and
    \begin{align*}
        \labs
        Z_{q_0}\lp x,\mu^1\rp
        -Z_{q_0}\lp x,\mu^2\rp
        \rabs
        &\leq
        \max_{i=1,2}
        \lp
        \frac1{q_0'}
        \lp Z_{q_0}\lp x,\mu^i\rp^{q_0'}\rp^{\frac1{q_0'}-1}\rp
        \labs
        Z_{q_0}\lp x,\mu^1\rp^{q_0'}
        -Z_{q_0}\lp x,\mu^2\rp^{q_0'}
        \rabs
        \\
        &\leq
        \frac1{q_0'}
        \lp\min_{i=1,2}{Z_{q_0}\lp x,\mu^i\rp}\rp^{1-q_0'}
        \int_{\TT^d}k(0,y)^{q_0'}dy\norminf{m^1-m^2}
        \\
        &\leq
        \frac1{q_0'}
        R^{\frac1{q_0}}
        \lp\int_{\TT^d}k(0,y)^{q_0'}dy\rp^{\frac1{q_0'}}
        \norminf{m^1-m^2}.
    \end{align*}
    The latter two chains of inequalities
    imply \eqref{eq:Vmudiff} with 
    $C_R=1+R^{1+\frac1{q_0'}}+\frac1{q_0'}R^2$.
\end{proof}

Here, we assume $\th\in(0,1)$.
Indeed,
$H$ admits an explicit form
when $\th=0$ or $\th=1$,
then checking
\ref{hypo:Hp}-\ref{hypo:mreg},
\ref{hypo:Hh}-\ref{hypo:Hpower},
\ref{hypo:Hinvert}-\ref{hypo:Hcontrac},
and \ref{hypo:Hlip}
is straightforward.

\begin{lemma}
    \label{lem:checkcrowd}
    Assumptions
    \ref{hypo:Hp},
    \ref{hypo:Hh}-\ref{hypo:Hpower},
    \ref{hypo:Hinvert},
    \ref{hypo:Hcontrac},
    and \ref{hypo:Hlip}
    are satisfied when $L$ is defined
    in \eqref{eq:defLcrowd}.
\end{lemma}

\begin{proof}
We define $\Lt$,
$\Ht$ and $\aaa$ as in the proof
of \ref{prop:exicrowd}.

    \emph{Checking \ref{hypo:Hp}, \ref{hypo:Hh} and \ref{hypo:Hx}.}

The Legendre transform of a function is convex,
therefore $H$ is convex with respect to $p$.
Since $L$ is strictly convex,
$H$ is differentiable with respect to $p$.
Moreover, $\aaa=-H_p$
thus $H_p$ is continuous by the Maximum theorem.
Then $H(x,p,\mu)=p\cdot H_p\lp x,p,\mu\rp
-L\lp x,-H_p\lp x,p,\mu\rp,\mu\rp$,
so $H$ is continuous.
Finally, $H$ is differentiable with respect
to $x$ by the envelop theorem
and
\begin{equation}
    H_x\lp x,p,\mu\rp
    =
    -L_x\lp x,-H_p\lp x,p,\mu\rp,\mu\rp,
\end{equation}
for $(x,p,\mu)\in\TT^d\times\RR^d\times\cP\lp\TT^d\times\RR^d\rp$.

Using the growth properties of $L$,
we can prove that there exists $C_0>0$
such that
    \begin{align}
        \label{eq:Hpboundcrowd}
        \labs H_p\lp t,x,p,\mu\rp\rabs
        &\leq
        C_0\lp 1+|p|^{q-1}
        +\LL_{q'}\lp\mu\rp\rp,
        \\
        \label{eq:Hboundcrowd}
        \labs H\lp t,x,p,\mu\rp\rabs
        &\leq
        C_0\lp 1+|p|^{q}
        +\LL_{q'}\lp\mu\rp^{q'}\rp,
        \\
        \label{eq:Hxboundcrowd}
        \labs H_x\lp t,x,p,\mu\rp\rabs
        &\leq
        C_0\lp 1+|p|^{q}
        +\LL_{q'}\lp\mu\rp^{q'}\rp,
    \end{align}
for any $\lp x,p,\mu\rp\in\TT^d\times\RR^d
\times\cP\lp\RR^d\times\RR^d\rp$.
We refer to \cite{MonoMFGC} Lemma $2.5$
for a complete proof.

One may prove that the function
$h:z\in\RR^d\mapsto|z|^{a'}\in\RR$ satisfies
$h(z)-h(y)-\nabla h(y)\cdot(y-x)\geq C_R^{-1}|y-z|^{\max(a',2)}$
for $y,z\in\RR^d$ such that
$|y|\leq R$,
$|z|\leq R$,
where $C_R>0$ is a constant.
This implies that
for $R>0$ there exists
$C_R>0$ a constant such that
$L$ satisfies
\begin{equation*}
    L\lp x,\aa^2,\mu\rp
    -L\lp x,\aa^1,\mu\rp
    -\lp\aa^2-\aa^1\rp
    \cdot L_{\aa}\lp x,\aa^1,\mu\rp
    \geq
    C_R^{-1}\labs \aa^2-\aa^1\rabs^{\max\lp q',2\rp},
\end{equation*}
for $\lp\aa^1,\aa^2,\mu\rp\in\RR^d\times\RR^d\times\cP_{\infty}\lp\TT^d\times\RR^d\rp$,
such that $\labs\aa^i\rabs\leq R$ and $\LL_{q_0}\lp\mu\rp\leq R$.
This implies
\begin{equation*}
    \lp\aa^2-\aa^1\rp
    \cdot\lp L_{\aa}\lp x,\aa^2,\mu\rp
    -L_{\aa}\lp x,\aa^1,\mu\rp\rp
    \geq
    2C_R^{-1}\labs \aa^2-\aa^1\rabs^{\max\lp q',2\rp}.
\end{equation*}
Take $p^i\in\RR^d$
and $\aa^i=-H_p\lp x,p^i,\mu\rp$,
$i=1,2$.
Recalling the conjugacy relation
$p^i=-L_{\aa}\lp x,\aa^i,\mu\rp$
we obtain
that $H_p$ is locally Hölder continuous
with respect to $p$.

\emph{Checking \ref{hypo:Hpower}.}

Take $(p,V)\in\RR^{2d}$
    and $\aaa=\aaa(p,V)$,
    the optimal control $\aaa$ satisfies
    \begin{equation}
        \label{eq:aadef}
        p
        =
        -D_{\aa}\Lt(\aaa,V)
        =
        -\th|\aaa-\llt V|^{a'-2}(\aaa-\llt V)
        -(1-\th)|\aaa|^{b'-2}\aaa,
    \end{equation}
If $(p,V)\neq (0,0)$,
    this implies
    \begin{equation}
        \label{eq:FPaa1}
        \aaa
        =
        \frac{-p+\llt \th|\aaa-\llt V|^{a'-2}V}
        {\th|\aaa-\llt V|^{a'-2}+(1-\th)|\aaa|^{b'-2}},
    \end{equation}
    and
    \begin{equation}
        \label{eq:FPaa2}
        \aaa-\llt V
        =
        \frac{-p+\llt (1-\th)|\aaa|^{b'-2}V}
        {\th|\aaa-\llt V|^{a'-2}+(1-\th)|\aaa|^{b'-2}}.
    \end{equation}
    From \eqref{eq:aadef}, we deduce that
    \begin{equation*}
        \th|\aaa-\llt V|^{a'-1}
        \geq
        \frac12|p|,
        \text{ or }
        (1-\th)|\aaa|^{b'-1}
        \geq
        \frac12|p|.
    \end{equation*}
    We recall that $\aaa=-H_p(p,V)$,
    hence
    \begin{align*}
        \Ht_p(p,V)\cdot p
        -\Ht(p,V)
        &=
        \Lt(\aaa,V) \\
        &=
        \frac{\th}{a'}|\aaa-\llt V|^{a'}
        + \frac{1-\th}{b'}|\aaa|^{b'}, \\
        &\geq
        \min
        \lp
        \frac{|p|^{a}}{2^aa'\th^{a-1}},
        \frac{|p|^{b}}{2^bb'(1-\th)^{b-1}}
        \rp,
    \end{align*}
    which implies \ref{hypo:Hpower}.

    \emph{Proof that $\aaa$ is differentiable
    with respect to $V$ at $(0,0)$.}

    Take $ V\in\RR^d$ that
    will eventually tend to $0$
    and $\aaa=\aaa(0, V)$.
    From \eqref{eq:aadef}
    we obtain
    \begin{equation*}
        0
        =
        \th\labs \aaa-\llt V\rabs^{a'-2}
        (\aaa-\llt V)
        +(1-\th)\labs \aaa\rabs^{b'-2}\aaa,
    \end{equation*}
    Let us recall inequalities \eqref{eq:estaa}
    and \eqref{eq:estaa2}.
    \begin{itemize}
        \item if $a>b$
            then $\frac{(a'-2)(b'-1)}{a'-1}<b'-2$,
            and we obtain the following expansion
            as $|V|$ tends to $0$,
            \begin{equation*}
                \aaa
                =
                \llt V
                +o(|V|).
            \end{equation*}

        \item if $a=b$
            we obtain,
            \begin{equation}
                \label{eq:esta=b}
                \aaa
                =
                \llt
                \frac{ \th^{a-1}}
                {\th^{a-1}+\lp1-\th\rp^{a-1}}
                V.
            \end{equation}

        \item if $a<b$
            then $\frac{(a'-2)(b'-1)}{a'-1}>b'-2$,
            and we obtain the following estimate
            as $|V|$ tends to $0$,
            \begin{equation*}
                \aaa
                =
                o(|V|).
            \end{equation*}
            Therefore the derivatives
            of $\aaa$ with respect to $V$
            in any of the above three cases
            are:
            \begin{equation}
                \label{eq:diffaa0}
                D_V \aaa(0,0)
                =
                \lc
                \begin{aligned}
                    &\llt I_d
                    &\text{ if } b<a \\
                    &\llt
                    \frac{ \th^{a-1}}
                    {\th^{a-1}+\lp1-\th\rp^{a-1}}
                    I_d
                    &\text{ if } b=a \\
                    &0 
                    &\text{ if } b>a. 
                \end{aligned}
                \right.
            \end{equation}
    \end{itemize}
    
    \emph{Proof that the operator norm of $D_V\aaa
    =\lp\partial_{V^j}\aaa^i\rp_{1\leq i,j\leq d}
    \in \RR^{d\times d}$
    is not larger than $\ll$.}

    Here,
    the norm of a square matrix $A\in\RR^{d\times d}$
    is defined by $\left\| A\right\|
        =
        \sup_{X\neq 0}\frac{\labs AX\rabs}{\labs X\rabs}$.
    Let us introduce
    \begin{align*}
        v_1
        &=
        \mathbf{1}_{\aaa-\llt V\neq0}
        \frac{\aaa-\llt V}
        {\labs \aaa-\llt V\rabs},
        \quad
        \quad
        &B
        =
        I_d+(a'-2)v_1v_1^T,
        \\
        v_2
        &=
        \mathbf{1}_{\aaa\neq0}
        \frac{\aaa}{|\aaa|},
        \quad
        \quad
        &C
        =
        I_d+(b'-2)v_2v_2^T.
    \end{align*}
    We recall that if $v_i\neq0$,
    then $v_iv_i^T$ is the
    orthogonal projection
    onto $\RR v_i$
    for $i=1,2$.
    
    If $\aaa=\llt V = 0$
    then $(p,V)=(0,0)$,
    we see on \eqref{eq:diffaa0} that
    $D_V\aaa$ is a positive semi-definite matrix 
    with eigenvalues in  $[-\ll,\ll]$.
    Therefore, we can now assume that
    $(\aaa,V)\neq(0,0)$.

    Let us assume temporarily that
        $a'\neq2, b'\neq 2,
        \aaa-\llt V\neq 0, \aaa\neq 0$.
        Then we differentiate the $i$-th
        component of \eqref{eq:aadef}
        with respect to $V^j$,
    \begin{multline*}
        0=\th\labs \aaa-\llt V \rabs^{a'-2}
        \lp \partial_{V^j}\aaa^i-\llt \dd_{i,j}\rp \\
        +\th(a'-2)\labs \aaa-\llt V \rabs^{a'-4} 
        \sum_{k=1}^d
        \lp \partial_{V^j}\aaa^k-\llt \dd_{k,j}\rp
        \lp \aaa^i-\llt V^i\rp
        \lp \aaa^k-\llt V^k\rp \\
        +(1-\th)|\aaa|^{b'-2}
        \partial_{V^j}\aa^i
        +(1-\th)(b'-2)|\aaa|^{b'-4}
        \sum_{k=1}^d
        \partial_{V^j}\aaa^k
        \aaa^i\aaa^k.
    \end{multline*}
    This implies
    \begin{equation*}
        0
        =
        \th\labs \aaa-\llt V\rabs^{a'-2}
        B\lp D_V\aaa-\llt I_d\rp
        +(1-\th)|\aa|^{b'-2}
        CD_V\aaa,
    \end{equation*}
    and thus
    \begin{equation}
        \label{eq:DVaaa}
        D_V\aaa
        =
        \llt\lb
        I_d
        +\frac{(1-\th)|\aaa|^{b'-2}}
        {\th\labs\aaa-\llt V\rabs^{a'-2}}
        B^{-1}C\rb^{-1}.
    \end{equation}
    We can check that this last equation
    holds in the general case
    for any $(\aaa,V)\neq(0,0),a',b'$.
    
    \begin{itemize}
    \item
        If $(a'-2)v_1=0$ (i.e. $B=I_d$)
        or $(b'-2)v_2=0$ (i.e. $C=I_d$),
        then \eqref{eq:DVaaa} yields that
        $D_V\aaa$ is a positive definite matrix
        with eigenvalues in $(-\ll,\ll)$.

    \item
        If $(a'-2)v_1\neq 0$ ,
        $(b'-2)v_2\neq 0$
        and $v_1,v_2$ are aligned,
        Then $B$ and $C$ commute
        and $B^{-1}C$ is a positive definite matrix.
        Then \eqref{eq:DVaaa} yields that
        $D_V\aaa$ is a positive definite matrix
        with eigenvalues in $(-\ll,\ll)$.

    \item
        The last case consists of assuming
        that $(a'-2)v_1\neq 0$ ,
        $(b'-2)v_2\neq 0$,
        and $v_1,v_2$ are linearly independent.
        We define $k$ by
        $k=\frac{(1-\th)|\aaa|^{b'-2}}
        {\th\labs\aaa-\llt V\rabs^{a'-2}}>0$.
        The two orthogonal subspaces
        ${\rm Span}(v_1,v_2)$ and $\{v_1,v_2\}^{\bot}$
        are stable by $D_V\aaa, B,C$.
        The restriction of $D_V\aaa$
        to $\{v_1,v_2\}^{\bot}$ is
        positive definite
        with eigenvalues in $(-\ll,\ll)$.

        Let us denote by $\At,\Bt,\Ct\in \cM_{2\times2}(\RR)$
        respectively the restriction of $D_V\aaa, B$
        and $C$ to ${\rm Span}(v_1,v_2)$.
        We notice that
        \begin{equation*}
            \Bt^{-1}
            =
            I_d+\lp(a'-1)^{-1}-1\rp v_1v_1^{\bot},
        \end{equation*}
        thus the eigenvalues of $\Bt^{-1}$
        are $1$ and $(a'-1)^{-1}\leq 1$
        since $a'\geq 2$.
        The eigenvalues of $\Ct$
        are $1$ and $(b'-1)\geq1$.
        Lemma \ref{lem:H'4} below yields that
        $M=(I_d+k\Bt^{-1}\Ct)(I_d+k\Ct\Bt^{-1})$
        is a positive definite matrix
        with eigenvalues not smaller
        than $1$.
        This implies
        \begin{align*}
            \|\At X\|^2
            &=
            \ll^2\left< M^{-1}X,X\right>\\
            &\leq
            \ll^2\|X\|^2.
        \end{align*}
    This concludes the proof that the norm of
    $D_V\aaa$ is not larger than $\ll$.
\end{itemize}

    \emph{Proof of \ref{hypo:Hcontrac}.}
    
    Take $(p,V^1,V^2)\in \RR^{3d}$
    and $\aaa^i=-\Ht_p\lp p,V^i\rp, i=1,2$,
    then
    \begin{align*}
        \labs
        \Ht_p\lp p,V_1\rp
        -\Ht_p\lp p,V_2\rp
        \rabs
        &\leq
        \sup_{s\in[0,1]}
        \lc
        \left\| D_V\aaa(p,sV_1+(1-s)V_2)\right\|
        \rc
        \labs V^1-V^2\rabs \\
        &\leq
        \ll
        \labs V^1-V^2\rabs.
    \end{align*}
    Combining the latter inequality 
    and \eqref{eq:Vmdiff},
    we conclude that \ref{hypo:Hcontrac} is satisfied.

    \emph{Proof of \ref{hypo:Hinvert}.}

    Let $(p,V)\in\RR^{2d}$,
    we take $\aaa=-H_p\lp p,V \rp$. 

    \begin{itemize}
        \item
            We suppose $b'\geq a'$,
            we make out two cases:
            the first case is when 
            $|\aaa|\leq |p|^{b-1}$;
            the second case is when 
            $|\aaa|> |p|^{b-1}=|p|^{\frac1{b'-1}}$
            which implies
            \begin{align*}
                |\aaa|
                &\leq
                \labs
                \frac{-p+\llt \th|\aaa-\llt V|^{a'-2}V}
                {\th|\aaa-\llt V|^{a'-2}+(1-\th)|\aaa|^{b'-2}}
                \rabs \\
                &\leq
                \frac{|p|}{(1-\th)|\aaa|^{b'-2}}
                +\ll|V| \\
                &\leq
                (1-\th)^{-1}|p|^{1-\frac{b'-2}{b'-1}}
                +\ll|V|,
            \end{align*}
            using \eqref{eq:FPaa1}.
            We recall that
            $1-\frac{b'-2}{b'-1}=b-1$,
            hence
            \begin{equation}
                \label{eq:H'invert_aux1}
                \labs \Ht_p\lp p,V\rp\rabs
                =
                |\aaa|
                \leq
                (1-\th)^{-1}|p|^{b-1}
                +\ll|V|. 
            \end{equation}

        \item
            We suppose that $b'<a'$,
            we make out two cases:
            the first case is when 
            $|\aaa-\llt V|\leq |p|^{a-1}$;
            the second case is when
            $|\aaa-\llt V|> |p|^{\frac1{a'-1}}$
            which implies
            \begin{align*}
                |\aaa|
                &\leq
                \labs
                \frac{-p+\llt \th|\aaa-\llt V|^{a'-2}V}
                {\th|\aaa-\llt V|^{a'-2}+(1-\th)|\aaa|^{b'-2}}
                \rabs \\
                &\leq
                \frac{|p|}
                {\th|\aaa-\llt V|^{a'-2}}
                +\ll|V| \\
                &\leq
                \th^{-1}|p|^{1-\frac{a'-2}{a'-1}}
                +\ll|V|, 
            \end{align*}
            where we used \eqref{eq:FPaa1}.
            From the equality
            $1-\frac{a'-2}{a'-1}=a-1$,
            we deduce
            \begin{equation}
                \label{eq:H'invert_aux2}
                \labs H_p\lp p,V\rp\rabs
                =
                |\aaa|
                \leq
                \th^{-1}|p|^{a-1}
                +\ll|V|. 
            \end{equation}
    \end{itemize}

    This concludes the proof of 
    \ref{hypo:Hinvert}.

    \emph{Proof of \ref{hypo:Hlip}.}

    We proved above that $\aaa$ is locally Lipschitz continuous
    with respect to $V$ and we recall that $\Lt$ is $C^1$.
    Therefore $\Ht$ is also locally Lipschitz with respect to $V$.
    This and \eqref{eq:Vmudiff} implies that \ref{hypo:Hlip} holds.
\end{proof}

    \begin{lemma}
        \label{lem:H'4}
        Let $B,C\in \cM_{2\times2}\lp\RR\rp$
        be two positive definite matrices
        with eigenvalues $(1,r)$ and $(1,s)$
        respectively, and $0<r\leq1, s\geq1$.
        Then for any $k>0$ the matrix $M$
        defined by
        \begin{equation*}
            M
            =
            I_d
            +k(BC+CB)
            +k^2BC^2B,
        \end{equation*}
        is positive definite
        with eigenvalues not smaller than $1$.
    \end{lemma}
    \begin{proof}
        We can assume that $B,C$
        have the following form:
        \begin{equation*}
            C
            =
            \begin{pmatrix}
                1 & 0 \\
                0 & s
            \end{pmatrix},
            \quad
            B
            =
            U
            \begin{pmatrix}
                1 & 0 \\
                0 & r
            \end{pmatrix}
            U^T,
            \text{ with }
            U\in \mathcal{O}_2\lp\RR\rp,
        \end{equation*}
        since the eigenvalues of $M$
        are invariant by taking the conjugate
        of $B$ and $C$ by the
        same orthogonal matrix.
        The same argument and noticing
        that $C$ commutes with
        $\begin{pmatrix}
            1 & 0 \\
            0 & -1
        \end{pmatrix}$,
        imply that
        we can assume that $U$
        admits a positive determinant,
        and thus we can write it as
        \begin{equation*}
            U=
            \begin{pmatrix}
                \cos \chi & \sin \chi \\
                -\sin \chi & \cos \chi
            \end{pmatrix},
        \end{equation*}
        with $\chi\in[0,2\pi)$.
        In this case, $M$ is given by
        \begin{align*}
            M
            &=
            I_d
            +k(BC+CB)
            +k^2BC^2B\\
            &\sim
            I_d
            +
            kU^T
            \begin{pmatrix}
                1 & 0 \\
                0 & s
            \end{pmatrix}
            U
            \begin{pmatrix}
                1 & 0 \\
                0 & r
            \end{pmatrix}
            +k
            \begin{pmatrix}
                1 & 0 \\
                0 & r
            \end{pmatrix}
            U^T
            \begin{pmatrix}
                1 & 0 \\
                0 & s
            \end{pmatrix}
            U
            +
            k^2
            \begin{pmatrix}
                1 & 0 \\
                0 & r
            \end{pmatrix}
            U^T
            \begin{pmatrix}
                1 & 0 \\
                0 & s^2
            \end{pmatrix}
            U
            \begin{pmatrix}
                1 & 0 \\
                0 & r
            \end{pmatrix}.
        \end{align*}
        We name $\Mt$ the matrix
        in the last line of the latter calculation,
        $M$ and $\Mt$ have the same eigenvalues.
        Let us compute $\Mt$
        \begin{equation*}
            \Mt
            =
            \begin{pmatrix}
                \cos^2\chi(1+k)^2+\sin^2\chi(1+ks)^2
                & -k(s-1)\lb 1+r+kr(1+s)\rb\cos\chi \sin\chi \\
                 -k(s-1)\lb 1+r+kr(1+s)\rb\cos\chi \sin\chi
                &\cos^2\chi(1+krs)^2+\sin^2\chi(1+kr)^2
            \end{pmatrix},
        \end{equation*}
        its trace is given by
        \begin{equation*}
            \text{tr}(\Mt)
            =
            \cos^2\chi(1+k)^2+\sin^2\chi(1+kr)^2
            +\cos^2\chi(1+krs)^2+\sin^2\chi(1+ks)^2,
        \end{equation*}
        and its determinant by
        \begin{align*}
            \det(\Mt)
            &=
            \!\begin{multlined}[t][10.5cm]
                (1+k)^2(1+krs)^2\cos^4\chi
                +(1+kr)^2(1+ks)^2\sin^4\chi \\
                +2(1+k)(1+kr)(1+ks)(1+krs)\cos^2\chi\sin^2\chi
            \end{multlined}
            \\
            &=
            \lb (1+k)(1+krs)\cos^2\chi
            +(1+kr)(1+ks)\sin^2\chi\rb^2.
        \end{align*}
        The eigenvalues of $\Mt$ are the roots of
        the following second-order polynomial function,
        \begin{equation*}
            X^2-\text{tr}(\Mt)X+\det(\Mt),
        \end{equation*}
        its smallest root is
        \begin{equation*}
            \frac12\lp \text{tr}(\Mt)
            -\sqrt{\text{tr}^2(\Mt)-4\det(\Mt)}\rp,
        \end{equation*}
        which is not smaller than $1$
        if and only if
        \begin{equation*}
            \text{tr}^2(\Mt)-4\det(\Mt)
            \leq
            \lp\text{tr}(\Mt)-2\rp^2.
        \end{equation*}
        Therefore, it is sufficient to check that tr$(\Mt)\leq\det(\Mt)+1$
        to conclude.
        We define the function $f:\RR\rightarrow\RR$ by
        \begin{multline*}
            f(x)
            =
            (1+k)^2(1+krs)^2x^2
            +(1+kr)^2(1+ks)^2(1-x)^2 \\
            +2(1+k)(1+kr)(1+ks)(1+krs)x(1-x)
            \cos^2\chi(1+k)^2 \\
            +\sin^2\chi(1+kr)^2
            +\cos^2\chi(1+krs)^2+\sin^2\chi(1+ks)^2+1.
        \end{multline*}
        This is a second-order polynomial in $x$
        with
        \begin{align*}
            f(0)
            &=
            \lp (1+kr)^2-1\rp
            \lp (1+ks)^2-1\rp
            \geq 0 \\
            f(1)
            &=
            \lp (1+k)^2-1\rp
            \lp (1+krs)^2-1\rp
            \geq 0, \\
            f''(x)
            &=
            2\lb (1+k)(1+krs)-(1+kr)(1+ks)\rb^2.
        \end{align*}
        If $(1+k)(1+krs)-(1+kr)(1+ks)=0$, then
        $f$ is linear and thus $f(x)\geq 0$ for all $x\in[0,1]$.

        If $(1+k)(1+krs)-(1+kr)(1+ks)\neq0$, then
        the minimum of this polynomial function on $\RR$
        is obtained at $x_{\min}$ defined as
        \begin{align*}
            x_{\min}
            &=
            \frac{(1+k)^2+(1+krs)^2-(1+ks)^2-(1-kr)^2}
            {2\lb (1+k)(1+krs)-(1+kr)(1+ks)\rb^2} \\
            &=
            \frac{(1-r^2)(1-s^2)k^2+2(1-r)(1-s)k}
            {2\lb (1+k)(1+krs)-(1+kr)(1+ks)\rb^2}
            \leq 0,\\
        \end{align*}
        since $0<r\leq 1, s\geq 1$ and $k>0$.
        Thus $f$ has no local minimum on $[0,1]$,
        then $f(x)\geq0$ for all $x\in[0,1]$
        since $f(0)\geq0$ and $f(1)\geq0$.

        Since $\det(\Mt)-\text{tr}(\Mt)+1=f(\cos^2\chi)\geq 0$,
        this concludes the proof of the lemma.
    \end{proof}

\end{appendix}

\end{document}